\RequirePackage{fix-cm}
\documentclass[smallextended]{svjour3}       % onecolumn (second format)
\smartqed  % flush right qed marks, e.g. at end of proof
\usepackage{graphicx}
\usepackage{amsmath}
\usepackage{amssymb}
\usepackage{subfig}
\usepackage{algorithm}
\usepackage{algpseudocode}
\usepackage{paralist}

\newcommand{\dist}{\mathop{\rm dist}\nolimits}

\def\argmin{\mathop{\rm argmin}\nolimits}
\newcommand{\amp}{\mathop{\:\:\,}\nolimits}
\newcommand{\Real}{\mathbb{R}}
\newcommand{\ba}{\boldsymbol{a}}
\newcommand{\bb}{\boldsymbol{b}}
\newcommand{\bc}{\boldsymbol{c}}
\newcommand{\bd}{\boldsymbol{d}}

\newcommand{\bq}{\boldsymbol{q}}

\newcommand{\bu}{\boldsymbol{u}}
\newcommand{\bv}{\boldsymbol{v}}

\newcommand{\bx}{\boldsymbol{x}}
\newcommand{\by}{\boldsymbol{y}}
\newcommand{\bz}{\boldsymbol{z}}
\newcommand{\bA}{\boldsymbol{A}}
\newcommand{\bB}{\boldsymbol{B}}

\newcommand{\bG}{\boldsymbol{G}}
\newcommand{\bH}{\boldsymbol{H}}
\newcommand{\bI}{\boldsymbol{I}}

\newcommand{\bM}{\boldsymbol{M}}

\newcommand{\bZero}{\boldsymbol{0}}

\newcommand{\eq}[1]{\begin{align*}#1\end{align*}}

\begin{document}

\title{An MM Algorithm for Split Feasibility Problems%\thanks{Grants or other notes
}

\author{Jason Xu \and
        Eric C. Chi \and
        Meng Yang \and 
        Kenneth Lange}

\authorrunning{J. Xu et al.} % if too long for running head

\institute{J. Xu \at
Department of Biomathematics, University of California, Los Angeles, CA 90095 \\
\email{jqxu@ucla.edu} \\
\and
              E. Chi \at
Department of Statistics, North Carolina State University, Raleigh, NC 27695 \\
\email{eric\_chi@ncsu.edu} \\
\and
              M. Yang \at
Department of Statistics, North Carolina State University, Raleigh, NC 27695 \\
\email{myang13@ncsu.edu} \\
\and
	     K. Lange  \at
Departments of Biomathematics, Human Genetics, and Statistics, University of California,
Los Angeles, CA 90095   \\
\email{klange@ucla.edu} \\
}

\date{Received: date / Accepted: date}

\maketitle

%% ----------------------------------------------------------------------
%% Abstract
%% ----------------------------------------------------------------------
\begin{abstract}
The classical multi-set split feasibility problem seeks a point in the intersection of finitely many closed convex domain constraints, whose image under a linear mapping also lies in the intersection of finitely many closed convex range constraints. Split feasibility generalizes important inverse problems including convex feasibility, linear complementarity, and regression with constraint sets. When a feasible point does not exist, solution methods that proceed by minimizing a proximity function can be used to obtain optimal approximate solutions to the problem. 
We present an extension of the proximity function approach that generalizes the linear split feasibility problem to allow for non-linear mappings. Our algorithm is based on the principle of majorization-minimization, is amenable to quasi-Newton acceleration, and comes complete with convergence guarantees under mild assumptions. Furthermore, we show that the Euclidean norm appearing in the proximity function of the non-linear split feasibility problem can be replaced by arbitrary Bregman divergences. We explore several examples illustrating the merits of non-linear formulations over the linear case, with a focus on optimization for intensity-modulated radiation therapy.
% 65K05 Mathematical programming methods
% 90C25 Convex programming
% 90C30 non-linear programming
% 62J02  General non-linear regression
\end{abstract}

%% ----------------------------------------------------------------------
%% Introduction
%% ----------------------------------------------------------------------
\section{Introduction}\label{sec:introduction}
In the split feasibility problem, one is given a smooth function $h(\bx): \Real^m \mapsto \Real^p$ 
and two closed convex sets $C \subset \Real^m$ and $Q \subset \Real^p$ \cite{CenElf1994}.  
One then seeks a point $\bx  \in\Real^m$ simultaneously satisfying $\bx \in C$ 
and $h(\bx) \in Q$. Instances of the split feasibility problem abound. The classical linear version of the split feasibility problem takes 
$h(\bx) = \bA \bx$ for some $m \times n$ matrix $\bA$ \cite{CenElf1994}. Other  typical examples of the split feasibility problem include constrained problems
such as finding $\bx \in C$ subject to $h(\bx) = \bb$, $\lVert h(\bx) -\bb\rVert \le r$, 
or $\bc \le h(\bx) - \bb \le \bd$. 
The multi-set version of the problem represents $C = \cap_i C_i$ and $Q = \cap_j Q_j$ as the
intersections of closed convex sets $C_i$ and $Q_j$. A variety of algorithms have 
been proposed for solving the linear split feasibility problem notably those in image reconstruction and intensity-modulated radiation therapy (IMRT) \cite{Byr2002,CenBorMar2006,CenElfKop2005,CenMotSeg2007,Xu2006}. The split feasibility problem has been studied in more general settings, for example in infinite dimensions \cite{Dingping2013,MasRei2007,Xu2010} and within the broader context of split inverse problems \cite{CenGibReich2012}.

Here we generalize a previously discovered algorithm \cite{CenElf1994} for the case when $h$ is linear to the case when $h$ is non-linear, using the majorization-minimization (MM) principle from computational statistics \cite{Lan2010,Lan2013}. Like the popular CQ algorithm for solving the split feasibility algorithm \cite{Byr2002}, the primary computations are projections onto the sets $C_i$ and $Q_j$. Our MM approach comes with global convergence guarantees under very mild assumptions and includes a first-order approximation to the Hessian that lends the algorithm to powerful acceleration schemes, in contrast to previous attempts. Such attempts at the non-linear problem include a projected gradient trust-region method \cite{gibali2014} that converges to local optima. This approach constitutes a sequential linearization of the problem, and is equivalent to the CQ algorithm in the case that $h$ is linear. An adaptive projection-type algorithm based on modifying the linear algorithm is proposed in \cite{li2013}. The objective function in the non-linear case becomes non-convex even when all constraint sets are closed and convex, and convergence results in \cite{li2013} rely on simply assuming convexity, a strong restriction that is furthermore not straightforward to verify. 

%Our approach instead produces an algorithm for solving a convex optimization problem, along with global convergence guarantees under very mild assumptions, by way of the MM principle.
Our point of departure is the proximity function \cite{CenElfKop2005}
\begin{eqnarray}
\label{eq:proximity_function1}
f(\bx) & = & \frac{1}{2} \sum_i v_i \dist(\bx,C_i)^2+
\frac{1}{2} \sum_j w_j \dist[h(\bx),Q_j]^2 , \label{proximity1}
\end{eqnarray}
where $\dist(\bx,\by)$ denotes the Euclidean distance between two vectors $\bx$ and $\by$. Note that $f(\bx)$ vanishes precisely when $\bx$ is split feasible. The weights $v_i$ and $w_j$ are
assumed to be positive with sum $\sum_i v_i + \sum_j w_j = 1$; nonuniform weights
allow some constraints to be stressed more than others. Alternative proximity functions have been employed in the split feasibility problem. For example, if we enforce a hard constraint that $\bx$ belongs to the intersection set $C$, we recover the proximity function behind the CQ algorithm for solving the split feasibility problem \cite{Byr2002}. In contrast to this and previous efforts toward the nonlinear problem, the proximity function defined in (\ref{eq:proximity_function1}) enables us to identify a point $\bx$ that is close to the desired constraints given by both $C$ and $Q$ even when the constraints cannot be satisfied simultaneously. Indeed, when there are no constraints on $h(\bx)$, the problem of minimizing the proximity function is closely related to the generalized Heron problem \cite{ChiLan2014,MorNam2011,MorNamSal2011,MorNamSal2012} of finding a
point that minimizes the distance to a collection of disjoint closed convex sets. 

The rest of our paper proceeds as follows. We begin by reviewing background material on the MM principle and deriving the MM algorithm. Next, we introduce generalized proximity functions based on Bregman divergences and derive an extension of the initial MM algorithm.  Convergence results for our MM algorithms are provided in the following section. Next, we illustrate how our MM algorithm can be applied to constrained regression problems for generalized linear models, with a focus on sparse regression. %in statistics as well as the non-linear complimentary problem in scientific computing.
Before concluding the paper, we highlight the computational advantages of the non-linear split feasibility formulation over its linear counterpart with a case study in computing radiation treatment plans for liver and prostate cancer.
%Its special cases include many important applications across mathematical disciplines: for instance, the two-set split feasibility problem with a singleton range constraint $Q = \{ \bb \}$ reduces to ridge regression in statistics.

%While the objective in (\ref{eq:proximity_function1}) is differentiable, there is no closed form solution in general for determining stationary points of (\ref{eq:proximity_function1}). Nonetheless, we introduce in this manuscript a simple general iterative framework for identifying stationary points of (\ref{eq:proximity_function1}) that relies on Euclidean projections onto the sets $C_i$ and $Q_j$. Often projections onto the sets $C_i$ and $Q_j$ of interest either have closed form solutions or can be computed efficiently. 
%When $h(\bx) = \bA\bx$, our MM algorithm reduces to the CQ algorithm \cite{Byr2002}. The MM strategy also opens the door to solving the split-feasibility problem using a proximity function based on Bregman divergences in place of the Euclidean distances. We will show the computational advantages of minimizing a Bregman based proximity function in an IMRT problem later.

%% ----------------------------------------------------------------------
%% MM Principle
%% ----------------------------------------------------------------------
\section{The MM Principle}

We briefly review the MM algorithm and its relationship to existing optimization methods before introducing our MM algorithm for solving the split-feasibility problem.
The MM algorithm is an instance of a broad class of methods termed sequential unconstrained minimization \cite{FiaMcC1990} and is closely related to the important subclass of sequential unconstrained minimization algorithms (SUMMA). These algorithms have attractive convergence properties \cite{Byr2008a,Byr2013,Byr2014}, and their relationship to MM algorithms is detailed in \cite{ChiZhoLan2014}.

The basic idea of an MM algorithm \cite{BecYanLan1997,LanHunYan2000} is to convert a hard optimization problem (for example, non-differentiable) into a sequence of simpler ones (for example, smooth). The MM principle requires majorizing the objective function $f(\by)$ above by a surrogate function $g(\by \mid \bx)$ anchored at the current point $\bx$.  Majorization is a combination of the tangency condition $g(\bx \mid \bx) =  f(\bx)$ and the domination condition $g(\by \mid \bx)  \geq f(\by)$ for all $\by \in \Real^m$.  The iterates of the associated MM algorithm are defined by
\begin{equation*}
  \bx_{k+1} := \underset{\by}{\arg \min}\; g(\by \mid \bx_{k}).
\end{equation*}
Because 
\begin{equation*}
  f(\bx_{k+1}) \leq g(\bx_{k+1} \mid \bx_{k}) \leq g(\bx_{k} \mid \bx_{k}) = f(\bx_{k}),
\end{equation*}
the MM iterates generate a descent algorithm driving the objective function downhill. Strict inequality usually prevails
unless $\bx_{k}$ is a stationary point of $f(\bx)$. We note that the expectation-maximization (EM) algorithm for maximum likelihood estimation with missing data is a well-known example of an MM algorithm, where the expected log-likelihood serves as the surrogate function \cite{dempster1977}.

Returning our attention to the split feasibility problem, we propose majorizing the proximity function $f(\bx)$ given in (\ref{eq:proximity_function1}) by the following surrogate function
\begin{eqnarray}
\label{eq:surrogate}
g(\bx \mid \bx_{k}) & = & \frac{1}{2} \sum_i v_i \lVert \bx-\mathcal{P}_{C_i}(\bx_{k})\rVert_2^2+
\frac{1}{2} \sum_j w_j \lVert h(\bx) - \mathcal{P}_{Q_j}[h(\bx_{k})]\rVert_2^2,
\end{eqnarray}
which is constructed from the Euclidean projections $\mathcal{P}_{C_i}(\bx)$ and $\mathcal{P}_{Q_j}(\by)$. The surrogate function satisfies the tangency condition $f(\bx_{k}) = g(\bx_{k} \mid \bx_{k})$
and the domination condition $f(\bx) \le g(\bx \mid \bx_{k})$ for all $\bx$.
As a consequence, the minimizer $\bx_{k+1}$ of $g(\bx \mid \bx_{k})$ satisfies
the descent property $f(\bx_{k+1}) \le f(\bx)$. In fact, validation of the descent property
merely requires $\bx_{k+1}$ to satisfy $g(\bx_{k+1} \mid \bx_{k}) \le g(\bx_{k} \mid \bx_{k})$.
Because $g(\bx \mid \bx_{k}) - f(\bx)$ is minimized by taking
$\bx = \bx_{k+1}$, the equality $\nabla g(\bx_{k} \mid \bx_{k}) = \nabla f(\bx_{k})$ 
always holds at an interior point $\bx_{k}$.

The stationary condition $\nabla g(\bx \mid \bx_{k}) = {\bf 0}$ requires the gradient
\begin{eqnarray}
\label{eq:stationary}
\nabla g(\bx \mid \bx_{k}) & = & \sum_i v_i [\bx-\mathcal{P}_{C_i}(\bx_{k})]+\sum_j w_j \nabla h(\bx)\{h(\bx) - \mathcal{P}_{Q_j}[h(\bx_{k})]\},
\label{gradient}
\end{eqnarray}
where $\nabla h(\bx) \in \Real^{n \times p}$ denotes the transpose of the first differential $dh(\bx) \in \Real^{p \times n}$. When $h$ is linear, identifying a stationary point reduces to solving a linear system. When $h$ is non-linear, we propose inexactly minimizing the surrogate $g(\bx \mid \bx_{k})$ by taking a single quasi-Newton step. To this end, we compute the second order differential $d^2 g(\bx \mid \bx_{k})$.
\begin{equation}
\label{eq:hessian}
\begin{split}
d^2 g(\bx \mid \bx_{k}) & = \sum_i v_i \bI_{n} +\sum_j w_j \nabla h(\bx) dh(\bx) \\
& + \sum_j w_j d^2 h(\bx)\{h(\bx) - \mathcal{P}_{Q_j}[h(\bx_{k})]\}.
\end{split}
\end{equation}
%By convention the first differential $dh(\bx)$ is the transpose of the gradient $\nabla h(\bx)$.
The second differential $d^2h(\bx)$ is properly interpreted as a tensor. 
The third sum in the expansion (\ref{eq:hessian}) vanishes whenever 
$h(\bx)$ is linear or when $h(\bx) \in Q$. This observation motivates the update rule 
\begin{eqnarray}
\label{eq:mm_update}
\bx_{k+1} & = & \bx_{k}-\Big[v \bI_{n} + w\nabla h(\bx_{k}) dh(\bx_{k}) \Big]^{-1} \nabla g(\bx_{k} \mid \bx_{k}) \nonumber \\
& = & \bx_{k}-\Big[v \bI_{n} + w\nabla h(\bx_{k}) dh(\bx_{k}) \Big]^{-1} \nabla f(\bx_{k}),
\end{eqnarray}
where $v$ is the sum of the $v_i$ and $w$ is the sum of the $w_i$. The update (\ref{eq:mm_update}) 
reduces to a valid MM algorithm when $h$ is a linear function. Otherwise, it constitutes an approximate
MM algorithm %due to dropping higher-order terms, and 
that can be safeguarded by step-halving to preserve the descent property.

\begin{algorithm}[t]
  \caption{MM-Algorithm for minimizing proximity function (Euclidean)}
  \label{alg:MM-Euclidean}  
  \begin{algorithmic}[1]
  	\State Given a starting point $\bx_0, \alpha \in (0, 1),$ and $\sigma \in (0,1)$
	\State $k \leftarrow 0$
	\Repeat
	\State $\nabla f_k \gets \sum_i v_i [\bx_{k}-\mathcal{P}_{C_i}(\bx_{k})]+\sum_j w_j \nabla h(\bx_{k})\{h(\bx) - \mathcal{P}_{Q_j}[h(\bx_{k})]\}$
	\State $\bH_{k} \gets v \bI_{n} + w\nabla h(\bx_{k}) dh(\bx_{k})$
	\State $\bv \gets - \bH_{k}^{-1}\nabla f_n$
	\State $\eta \gets 1$	
	\While{$f(\bx_{k} + \eta\bv) > f(\bx_{k}) + \alpha \eta \nabla f_n^t\bx_{k}$}
		\State $\eta \gets \sigma \eta$ 
	\EndWhile
	\State $\bx_{k+1} \gets \bx_{k} + \eta\bv$
	 \State $k \leftarrow k+1$
	 \Until{convergence}
  \end{algorithmic}
\end{algorithm}

In essence, the update (\ref{eq:mm_update}) approximates one step of
Newton's method applied to the surrogate \cite{Lan1995}.  Dropping the sum involving
$d^2h(\bx_{k})$ is analogous to the transition from Newton's method to 
the Gauss-Newton algorithm in non-linear regression. This maneuver avoids 
computation of second derivatives while guaranteeing that the approximate second differential is positive definite. 
As the algorithm approaches split feasibility, 
the residual terms $h(\bx) - \mathcal{P}_{Q_j}[h(\bx_{k})]$ nearly vanish, further
improving the approximation and accelerating convergence. As the overall rate of convergence of any MM algorithm is linear, one step of a quasi-Newton method suffices; additional computations of inner iterations required to find the minimum of
$g(\bx \mid \bx_{k})$ are effectively wasted effort. Algorithm~\ref{alg:MM-Euclidean} summarizes the MM algorithm for minimizing (\ref{eq:proximity_function1}) 
using step-halving based on satisfying the Armijo condition.

Note that the algorithm requires inversion of a linear system involving the $n$-by-$n$ matrix $\bH_{k}$, requiring $\mathcal{O}(n^3)$ flops. 
When $p \ll n$ in the mapping $h : \Real^n \rightarrow \Real^p$, we can invoke the Woodbury formula, which tells us that the inverse of 
\begin{eqnarray*}
\bH_{k} & = & v \bI_{n} +w \nabla h(\bx_{k}) dh(\bx_{k})
\end{eqnarray*}
can be expressed as
\begin{eqnarray*}
\bH_{k}^{-1} & = & \frac{1}{v}\bI_{n} - \frac{w}{v^{2}}\nabla h(\bx_{k}) \Big(\bI_{p} + \frac{w}{v} dh(\bx_{k})\nabla h(\bx_{k})\Big)^{-1}dh(\bx_{k}).
\end{eqnarray*}
In this case, 
each update requires solving the substantially smaller $p$-by-$p$ linear system of equations if the latter form is used. Savings can be non-trivial as the dominant computation in the latter form is the matrix-matrix product $dh(\bx_{k})\nabla h(\bx_{k})$ which requires $\mathcal{O}(np^2)$ flops.
%In the linear case, the small inverse $(\bI_q+v^{-1}w\bA\bA^t)^{-1}$ can be
%stored and reused at each iteration. The remaining steps of the MM algorithm
%require just vector and matrix times vector operations.

%% ----------------------------------------------------------------------
%% Illustrative Example
%% ----------------------------------------------------------------------

\begin{figure}[t]
\centering
\includegraphics[scale=0.7]{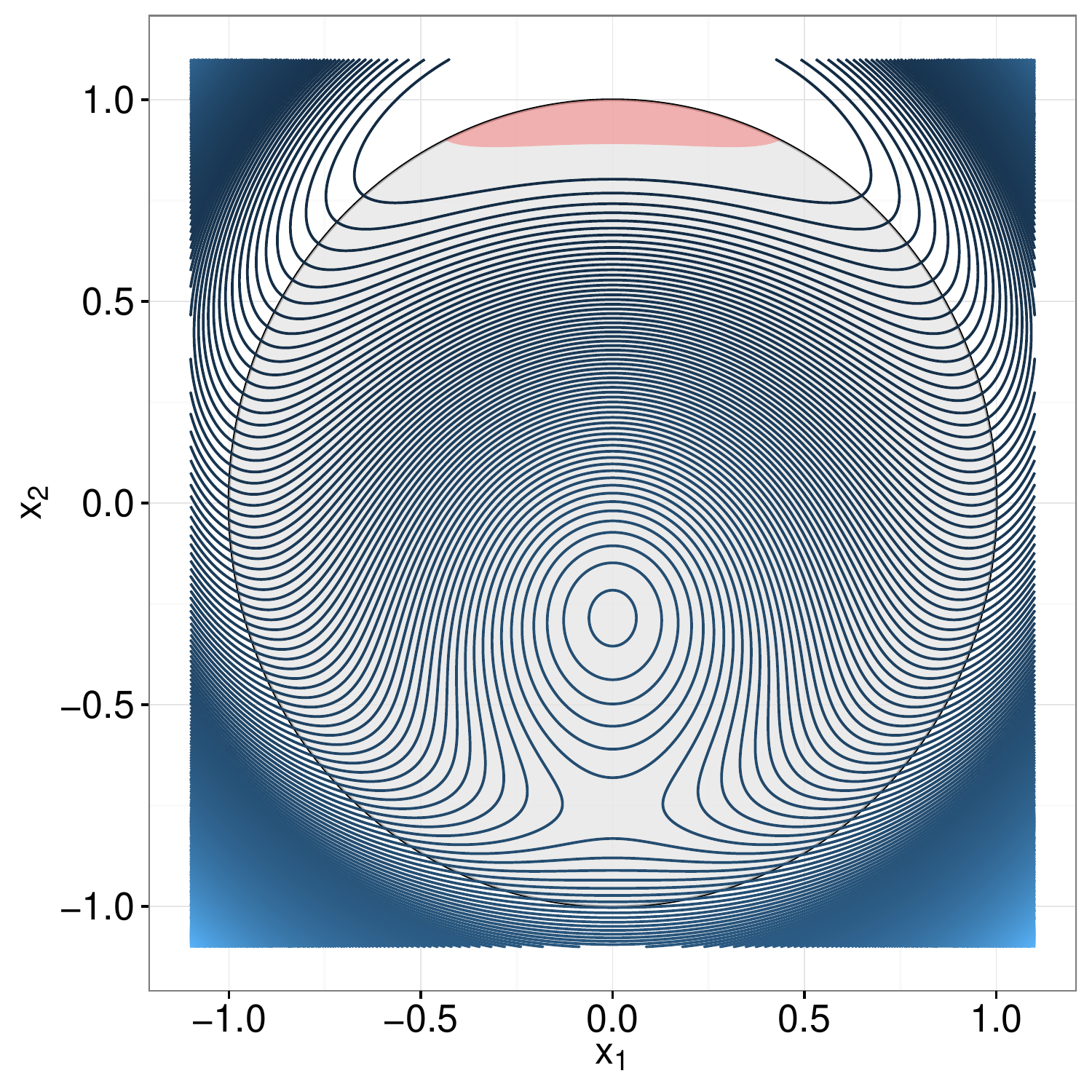}
\caption{Contour plot of the proximity function of a non-linear split feasibility problem.}
\label{fig:proximity_illustrative_ex}
\end{figure}

\paragraph{\bf Illustrative Example:} We present a small toy problem to demonstrate the MM algorithm and to motivate an important refinement.
Consider the non-linear split feasibility problem with domain set $C = \{\bx \in \Real^2 : \lVert \bx \rVert_2 \leq 1\}$, range set $Q = \{ \by \in \Real^3 : \lVert \by - \bd \rVert \leq 1\}$ where $\bd = (0,1.8,3)^t$, and mapping $h : \Real^2 \rightarrow \Real^3$. Figure~\ref{fig:proximity_illustrative_ex} shows the level sets of the proximity function
for this problem. The light grey region denotes $C$ and
the red region denotes $C \cap h^{-1}(Q)$. There are two things to observe. 
First, the proximity function is non-convex, and second, the proximity function is very flat around the split feasible region. In the non-linear split feasibility problem the proximity function in general is non-convex; we will give a useful example later in which the proximity function is convex. Nonetheless, in the most general case, we will see that only convergence to a stationary point of $f$ is guaranteed. The flatness of the proximity function near the split feasible region is important to note since it will affect the local convergence rate.

The left panel of Figure~\ref{fig:MM_sequence} shows the MM iterate sequence generated from six different starting points. All sequences converge to a split feasible point. In all cases except for the sequence starting at (0,0), the iterate sequence makes great initial progress only to slow down drastically as the proximity function becomes flatter around the split feasible region.

\begin{figure}[t]
\centering
\includegraphics[scale=0.405]{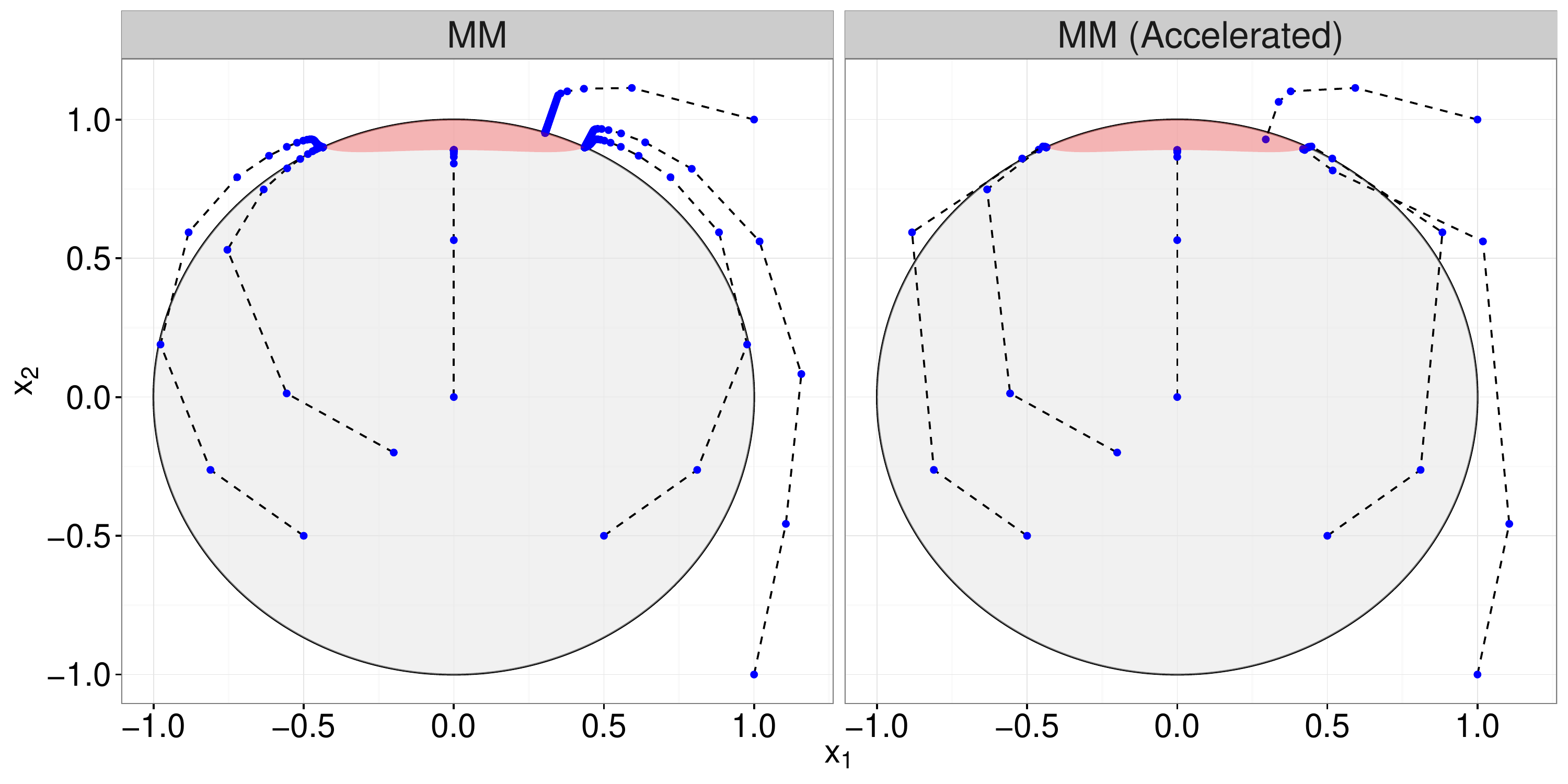}
\caption{The MM iterate sequence for a non-linear split feasibility problem initialized at six different points. On the left is the unaccelerated MM sequence. On the right is an accelerated sequence using 2 secants.}
\label{fig:MM_sequence}
\end{figure}

%% ----------------------------------------------------------------------
%% quasi-Newton acceleration
%% ----------------------------------------------------------------------
As this example demonstrates, MM algorithms are often plagued by a slow rate of convergence in a neighborhood of a minimum point. To remedy this situation, we employ quasi-Newton acceleration on the iteration sequence. Details on the acceleration are given in \cite{ZhoAleLan2011}.
Returning to our toy example, we see in the right panel of Figure~\ref{fig:MM_sequence} a tremendous reduction in the number of iterates until convergence using quasi-Newton acceleration with two secant approximations.

%% ----------------------------------------------------------------------
%% Generalization to Bregman Divergences
%% ----------------------------------------------------------------------
\section{Generalization to Bregman Divergences}

Note that the key property of the proximity function given in (\ref{eq:proximity_function1}) is that it attains its global minimum of zero at a point $\bx$ if and only if $\bx$ is split feasible. Thus, we may wish to consider employing alternatives to the Euclidean distance as measures of closeness in the proximity function, that maintain this key property. A natural alternative is the Bregman divergence, which we review now. Indeed, we will see in our IMRT case study that employing a Bregman divergence will result in improved computational speeds.
%Furthermore, transforming the geometry of the proximity function by using a Bregman divergence allows us to extend regression under the split feasibility framework to generalized linear models.

Let $\phi(\bu)$ be a strictly convex, twice differentiable function.  The Bregman divergence \cite{bregman1967} between $\bu$ and $\bv$ with respect to $\phi$ is defined as
\begin{eqnarray}
\label{eq:bregman}
D_\phi(\bv, \bu) & = & \phi(\bv) - \phi(\bu)- d\phi(\bu)(\bv - \bu).
\end{eqnarray}
Note that the Bregman divergence (\ref{eq:bregman}) is a convex function of its first argument $\bv$ that majorizes 0. It is natural to consider the Bregman divergence since it encompasses many useful measures of closeness. For instance, the convex
functions $\phi_1(\bv) = \lVert \bv \rVert_2^2$, $\phi_2(\bv)= -\sum_i \log v_i$, 
%$\phi_3(\by) = \sum_i y_i \log y_i$, 
and $\phi_3(\bv) = \bv^t \bM \bv$ generate the Bregman divergences 
\begin{eqnarray*}
D_{\phi_1}(\bv, \bu) & = & \lVert \bv - \bu \rVert_2^2 \\
D_{\phi_2}(\bv, \bu) & = & \sum_i \Big[{v_i \over u_i}-\log \Big({v_i \over u_i}\Big) -1\Big] \\
%D_{\phi_3}(\by \mid \bx) & = & \sum_i y_i \log \Big({y_i \over x_i}\Big) - \sum_i (y_i-x_i) \\
D_{\phi_3}(\bv, \bu) & = & (\bv - \bu)^t \bM (\bv - \bu). 
\end{eqnarray*}
The matrix $\bM$ in the definition of $\phi_3(\bv)$ is assumed to be positive definite.

Note that the Bregman divergence is not a metric, as it is not symmetric in general. For instance, while $D_{\phi_1}$ and $D_{\phi_3}$ are metrics, $D_{\phi_2}$ is not a metric.
Nonetheless, it has the metric-like property that the $D_{\phi}(\bv, \bu)$ is nonnegative for all $\bu$ and $\bv$ and vanishes if and only if $\bu = \bv$ when $\phi$ is strictly convex as we have assumed. We will see shortly that this property enables us to employ Bregman divergences in the proximity function in place of the Euclidean distance \cite{censor1997parallel}. %But first we define the Bregman projection.

Before introducing the generalized proximity function, we recall that the Bregman projection of a point $\bu$ onto a nonempty closed convex set $C$ with respect to $\phi$ is the point $\mathcal{P}^\phi_C(\bu)$ defined as
\begin{eqnarray*}
\mathcal{P}^\phi_C(\bu) & = & \underset{\bv \in C}{\arg\min} \; D_\phi(\bv, \bu).
\end{eqnarray*}
Under suitable regularity conditions, the Bregman projection exists. Moreover, the Bregman projection 
is unique when it exists by strict convexity of $\phi(\bu)$. 

Let $\phi: \Real^n \rightarrow \Real$ and $\zeta: \Real^p \rightarrow \Real$ be strictly convex, twice differentiable functions.
Note that $\mathcal{P}^\phi_C(\bx) = \bx$, or equivalently $D_\phi \left (\mathcal{P}^\phi_C(\bx), \bx \right) = 0$,
exactly when $\bx \in C$. The same applies to range sets: $\mathcal{P}^\phi_Q [h(\bx)] = h(\bx)$, or equivalently $D_\zeta \left(\mathcal{P}^\zeta_Q[h(\bx)], h(\bx) \right) = 0$,
precisely when $h(\bx) \in Q$. 
Consequently, a point $\bx$ is split feasible if and only if  for all $i$ and $j$,
\begin{eqnarray*}
D_\phi \left (\mathcal{P}^\phi_{C_i}(\bx), \bx \right) = 0 \quad\quad\text{and}\quad\quad D_\zeta \left ( \mathcal{P}^{\zeta}_{Q_j}[ h(\bx)], h(\bx)\right)= 0.
\end{eqnarray*}

This observation motivates the following analogue of the
proximity function $f(\bx)$ from (\ref{eq:proximity_function1})
\begin{eqnarray}
\label{eq:proximity_function2}
\tilde{f}(\bx) & = & \sum_i v_i D_\phi \left (\mathcal{P}^\phi_{C_i}(\bx) , \bx \right) + 
\sum_j w_j D_\zeta \left ( \mathcal{P}^{\zeta}_{Q_j}[ h(\bx)] , h(\bx)\right ).
\end{eqnarray}
Note that the proximity function $\tilde{f}(\bx)$ in (\ref{eq:proximity_function2}) coincides with the original proximity function $f(\bx)$ in (\ref{eq:proximity_function1}) when $\phi$ and $\zeta$ are the squared Euclidean norms.

The proximity function $\tilde{f}(\bx)$ can be majorized in a similar fashion to our original proximity function $f(\bx)$ in (\ref{eq:proximity_function1}).
If we abbreviate $\by^i_{k} = \mathcal{P}^\phi_{C_i}(\bx_{k})$ and $\bz^j_{k} = \mathcal{P}^\zeta_{Q_j}[h(\bx_{k})]$,
then the surrogate function
\begin{eqnarray*}
\tilde{g}(\bx \mid \bx_{k}) & = & \sum_i v_i D_\phi\left(\by^i_{k} , \bx \right)+
\sum_j w_j D_\zeta\left(\bz^j_{k}, h(\bx) \right) 
\end{eqnarray*}
majorizes $\tilde{f}(\bx)$. This follows from the definition of the Bregman projection:
\begin{eqnarray*}
D_\phi \left (\mathcal{P}^\phi_{C_i}(\bx) , \bx \right ) & \leq & D_\phi \left (\mathcal{P}^\phi_{C_i}(\bx_{k}) , \bx \right ) \\
D_\zeta \left (\mathcal{P}^\zeta_{Q_j}[h(\bx)] , h(\bx) \right ) & \leq & D_\zeta \left (\mathcal{P}^\zeta_{Q_j}[h(\bx_{k})] , h(\bx) \right ).
\end{eqnarray*}
The MM principle suggests that we minimize $\tilde{g}(\bx \mid \bx_{k})$.
Since
\begin{eqnarray*}
\nabla D_\phi \left (\mathcal{P}^\phi_{C_i}(\bx_{k}) , \bx \right) & = & d^2 \phi(\bx) \left (\bx - \mathcal{P}^\phi_{C_i}(\bx_{k}) \right ) \\
\nabla D_\zeta \left (\mathcal{P}^\zeta{Q_j}[h(\bx_{k})] , h(\bx) \right) & = & d^2 \zeta[h(\bx)] \left (h(\bx) - \mathcal{P}^\phi_{Q_j}[h(\bx_{k})] \right ),
\end{eqnarray*}
a brief calculation produces the gradient of the majorization $\tilde{g}(\bx \mid \bx_{k})$ with respect to its first argument.
\begin{eqnarray*}
\nabla \tilde{g}(\bx \mid \bx_{k}) = \sum_i v_i d^2 \phi(\bx) (\bx - \by^i_{k} )
+ \sum_j w_j \nabla h(\bx) d^2 \zeta[ h(\bx) ][ h(\bx) -\bz^j_{k}].
\end{eqnarray*}
Once again, the tangency condition $\nabla \tilde{g}(\bx_{k} \mid \bx_{k}) = \nabla \tilde{f}(\bx_{k})$ holds. 
To obtain an improved point, we consider the full second differential
\begin{eqnarray*}
d^2 \tilde{g}(\bx \mid \bx_{k}) & = & \sum_i v_i \bA_i + \sum_j w_j \bB_j,
\end{eqnarray*}
where
\begin{eqnarray*}
\bA_i & = & d^3 \phi(\bx)( \bx - \by^i_{k} ) + d^2 \phi(\bx) \\
\bB_j & = & d^2 h(\bx) d^2 \zeta[ h(\bx) ] ( h(\bx) - \bz^j_{k} ) 
  + \nabla h(\bx) d^3 \zeta[ h(\bx)] d h(\bx) ( h(\bx) - \bz^j_{k} ) \\ 
 & + & \nabla h(\bx) d^2 \zeta[h(\bx)] d h(\bx).
\end{eqnarray*}
While this expression for the second differential is notably more unwieldy compared to its counterpart in the Euclidean case, the terms containing differences between $\bx$ or $h(\bx)$ and their projections---most of the terms above---vanish when $\bx \in C$ or $h(\bx) \in Q$ and become negligible near a feasible solution, greatly simplifying computations. Thus, we approximate the Hessian as follows:
\begin{eqnarray*}
d^2 \tilde{g}(\bx \mid \bx_{k}) & \approx & v \tilde{\bA} + w\tilde{\bB},
\end{eqnarray*}
where
\begin{eqnarray*}
\tilde{\bA} = d^2 \phi(\bx) \quad\quad \text{and} \quad\quad \tilde{\bB} = \nabla h(\bx) d^2 \zeta[h(\bx)] d h(\bx).
\end{eqnarray*}
This approximation motivates the update
\begin{eqnarray}
\label{eq:mm_update2}
\bx_{k+1} & = & \bx_{k}-\Big[v d^2 \phi(\bx_{k}) + w\nabla h(\bx_{k}) d^2 \zeta[h(\bx_{k})] d h(\bx_{k}) \Big]^{-1}
\nabla \tilde{f}(\bx_{k}).
\end{eqnarray}
Because the matrices $d^2\phi(\bx_{k})$ and $d^2 \phi[h(\bx_{k})]$ are positive
definite, the new algorithm enjoys the descent property if some form of 
step-halving is instituted.  Algorithm~\ref{alg:MM-Bregman} summarizes the MM algorithm for minimizing (\ref{eq:proximity_function2}) using step-halving based on satisfying the Armijo condition. Note that Algorithm~\ref{alg:MM-Bregman} reduces to Algorithm~\ref{alg:MM-Euclidean} when we take $\phi$ and $\zeta$ to be the squared Euclidean norms.

\begin{algorithm}[t]
  \caption{MM-Algorithm for minimizing proximity function (Bregman)}
  \label{alg:MM-Bregman}  
  \begin{algorithmic}[1]
  	\State Given a starting point $\bx_0, \alpha \in (0, 1),$ and $\sigma \in (0,1)$
	\State $k \leftarrow 0$
	\Repeat
	\State $\nabla \tilde{f}_{k} \gets \sum_i v_i d^2 \phi(\bx_{k}) (\bx_{k} - \mathcal{P}^\phi_{C_i}(\bx_{k}) )
+ \sum_j w_j \nabla h(\bx_{k}) d^2 \zeta[ h(\bx_{k}) ][ h(\bx_{k}) - \mathcal{P}^\zeta_{Q_j}[h(\bx_{k})]] $
	\State $\bH_{k} \gets v d^2 \phi(\bx_{k}) + w\nabla h(\bx_{k}) d^2 \zeta[h(\bx_{k})] d h(\bx_{k})$
	\State $\bv \gets - \bH_{k}^{-1}\nabla \tilde{f}_k$
	\State $\eta \gets 1$
	\While{$\tilde{f}(\bx_{k} + \eta\bv) > \tilde{f}(\bx_{k}) + \alpha \eta \nabla \tilde{f}_n^t\bx_{k}$}
		\State $\eta \gets \sigma \eta$ 
	\EndWhile
	\State $\bx_{k+1} \gets \bx_{k} + \eta\bv$
	 \State $k \leftarrow k+1$
	 \Until{convergence}
  \end{algorithmic}
\end{algorithm}

We remark that one can take Bregman projections with respect to divergences generated by set-specific functions. For example, we could use the Bregman divergence with respect to a function $\phi_i(\bx)$ for computing the proximity of $\bx$ to the set $C_i$. Similarly, we could use the Bregman divergence with respect to a function $\zeta_j$ for computing the proximity of $h(\bx)$ to the set $Q_j$. Whenever the functions $\phi_i(\bx)$ and $\zeta_j(h(\bx))$ are parameter separated, the second differentials $d^2\phi_i(\bx_{k})$ $d^2\zeta_j(h(\bx_{k}))$ are diagonal, and application of Woodbury's formula is straightforward.

%% ----------------------------------------------------------------------
%% Computing the Bregman Projections
%% ----------------------------------------------------------------------
%\subsection{Computing the Bregman Projections}
%\label{sec:Bregman_Projection}
\paragraph{\bf Computing the Bregman Projections:} Existence of the Bregman projection does not imply that it is always practical to compute. Here we describe a general way to obtain the Bregman projection of a point onto a hyperplane or half-space. We focus on this case for two reasons. First, computation in these two cases reduce to a one-dimensional optimization problem at worst, and admits closed form solutions for some choices of $\phi$. Second, closed convex sets can be can be approximated with a tunable degree of accuracy with a well chosen intersection of half-spaces, or polytope. Other simple convex sets, such as affine subspaces and cones can be computed using similar algorithmic primitives \cite{lorenz2014}. 

The Bregman projections will require computations involving Fenchel conjugates;
recall that the Fenchel conjugate of a function $\phi$ is
\begin{eqnarray*}
\phi^\star(\by) & = & \underset{\bx}{\sup}\; \by^t \bx - \phi(\bx).
\end{eqnarray*}
 %We will require the Fenchel conjugate to be Lipschitz continuous. A sufficient condition to meet this requirement, is the strong convexity of the function $\phi$. 
%Recall that a differentiable function $\phi$ is strongly convex with modulus $\alpha > 0$ if
%\begin{eqnarray*}
%\phi(\by) & \geq & \phi(\bx) + \nabla \phi(\bx)^t(\by - \bx) + \frac{\alpha}{2}\lVert \by - \bx \rVert_2^2,
%\end{eqnarray*}
%for all $\bx$ and $\by$. When $\phi$ is strongly convex, its Fenchel conjugate has a Lipschitz continuous gradient with parameter $\alpha^{-1}$.
%
The Bregman projection of a point $\bx$ onto the hyperplane $HP(\ba, c) = \left\{ \bz : \ba^t\bz = c \right\}$ with nonzero $\ba \in \Real^p$ is given by
\begin{equation}\label{eq:bregproj}
 \mathcal{P}^\phi_{HP(\ba, c)}(\bx) = \nabla \phi^\star ( \nabla \phi(\bx) - \gamma \ba ) ,
\end{equation}
where $\gamma$ is a scalar satisfying
\begin{eqnarray*}
\gamma \amp := \amp \underset{\widetilde{\gamma} \in \Real}{\arg\min}\; \phi^\star ( \nabla \phi(\bx) - \widetilde\gamma \ba ) + \widetilde\gamma c.
\end{eqnarray*}

Projecting onto a half-space follows similarly from this relation. Denote the half-space $HS(\ba, c) := \left\{ \bz : \ba^t \bz \leq c \right\}$. Then $\gamma=0$ when $\bx \in HS(\ba,c)$, but otherwise $\gamma>0$, and $\mathcal{P}^\phi_{HP(\ba, c)}(\bx) = \mathcal{P}^\phi_{HS(\ba, c)}(\bx)$.
It is straightforward to verify, for instance, that this coincides with the Euclidean projection when $\phi(\bx) = \phi^\star(\bx) = \frac{1}{2} \lVert \bx \rVert_2^2$.

To illustrate this calculation consider the Kullback-Leibler (KL) divergence or relative entropy 
\begin{eqnarray*}
KL(\by, \bx) = D_{\phi}(\by, \bx)  =  \sum_j y_j \log \Big({y_j \over x_j}\Big) + x_j - y_j,
\end{eqnarray*}
generated by the negative entropy function 
$$\phi(\bx) = \sum_j x_j \log x_j, \qquad \bx \in \Real_+.$$  %, with parameter $\beta=1$.
%Note that $\phi$ is strongly convex, and 
%Strictly convex --> Fenchel conjugate continuously differentiable
In this case we have $\nabla \phi(\bx) = \log \bx + 1$, $\phi^\star(\bz) = \sum_j e^{z_j - 1}$, and $\nabla \phi^\star(\bz) = e^{\bz - 1}$. Using (\ref{eq:bregproj}), one can find that the Bregman projection is given by the vector $\mathcal{P}^\phi_{HP(\ba, c)}(\bx) = \bx e^{-\gamma \ba}$, where $\gamma$ minimizes the quantity $\sum_j x_j e^{-\widetilde\gamma a_j} + \widetilde\gamma c$. While there is no closed form for $\gamma$ in this example, the projection can be computed numerically by solving a simple one-dimensional optimization problem, with second order information readily available for efficient implementations of Newton's method. 

The KL divergence arises in a wide variety of applications, notably in image denoising and tomography, enabling entropy-based reconstruction from non-negative image data \cite{byrne1993}.
Notice that $\mathcal{P}^\phi_{Q(\ba, c)}(\bx)$ results in a multiplicative update of $\bx$ and therefore preserves positivity, a desirable feature in these settings. Indeed, the multiplicative algebraic reconstruction technique, a classical method in positron emission tomography (PET) \cite{gordon1970}, was later analyzed as a sequential Bregman projection algorithm \cite{byrne1999}. 
We note that the KL divergence is a special case of the important class of $\beta$-divergences with parameter $\beta=1$, which we will discuss in our IMRT case study in Section~\ref{sec:imrt}. When intermediate values within iterative procedures may take non-positive values, as is the case in our MM algorithm, $\beta$-divergences with $\beta > 1$ can be considered and are well-defined when $\bx$ assumes negative entries. This class of divergences has found success in applications such as non-negative matrix factorization \cite{fevotte2011}.

%% ----------------------------------------------------------------------
%% Convergence Analysis
%% ----------------------------------------------------------------------
\section{Convergence Analysis}\label{sec:converge}

We now show that the set of limit points of the MM algorithm belong to the set of stationary points of the proximity function. Moreover, the set of limit points is compact and connected. We prove convergence under both the Euclidean based proximity function in (\ref{eq:proximity_function1}) and its Bregman generalization in (\ref{eq:proximity_function2}). We make the following assumptions:
\begin{itemize}
\item[A1.] The proximity function is coercive, namely its sublevel sets are compact.
\item[A2.] The gradient of the proximity function is $L$-Lipschitz continuous, or in other words is Lipschitz differentiable.
\item[A3.] The approximate Hessian of the proximity function $\bH(\bx)$ maps $\Real^n$ continuously into the space of positive definite $n \times n$ matrices, namely the matrix $\bH(\bx) =
 v \bI_{n} + w\nabla h(\bx) dh(\bx)$ is continuous in the case of the Euclidean based proximity function, or $\bH(\bx) = v d^2 \phi(\bx) + w\nabla h(\bx) d^2 \zeta[h(\bx)] d h(\bx)$ is continuous in the case of the Bregman generalization.
\end{itemize}

%As a prelude, we review some consequences of the key concept of Lipschitz continuity.
%\begin{proposition}
%If $\nabla f(\bx)$ is $L$-Lipschitz continuous, then for all $\bx$ and $\by$
%\begin{eqnarray*}
%f(\by) & \leq & f(\bx) + df(\bx)[\by - \bx] + \frac{L}{2} \lVert \by - \bx \rVert_2^2.
%\end{eqnarray*}
%\end{proposition}
%\begin{proof}
%Fix $\bx$ and $\by$. Let $g(t) = f[\bx + t(\by-\bx)]$. Then by the fundamental theorem of calculus
%\begin{eqnarray*}
%g(1) & = & g(0) + \int_0^1g'(t) dt \\
%f(\by) & = & f(\bx) + \int_0^1 \nabla f[\bx + t(\by-\bx)]^t [\by-\bx] dt \\
% & = & f(\bx) + df(\bx)[\by - \bx] + \int_0^1 [\nabla f[\bx + t(\by-\bx)] - \nabla f(\bx)] ]^t [\by-\bx] dt \\
% & \leq & f(\bx) + df(\bx)[\by - \bx] + \int_0^1 \|\nabla f[\bx + t(\by-\bx)] - \nabla f(\bx)\|_2
% \| \by-\bx \|_2 dt \\
%  & \leq & f(\bx) + df(\bx)[\by - \bx] + \int_0^1 Lt \| \by - \bx \|_2^2 dt \\
%  & = & f(\bx) + df(\bx)[\by - \bx] + \frac{L}{2} \| \by - \bx \|_2^2.\\
%\end{eqnarray*}
%\end{proof}

%Throughout the remainder of our discussion on convergence we assume that $f(\bx)$ is coercive and that $\nabla f(\bx)$ is $L$-Lipschitz continuous.
We begin by showing that step-halving under the Armijo condition is guaranteed to terminate under finitely many backtracking steps. We first review the Armijo condition: suppose $\bv$ is a descent direction at $\bx$ in the sense that $df(\bx)\bv < 0$. The Armijo condition chooses a step size $\eta$ such that
\begin{eqnarray*}
f(\bx + t\bv) \leq f(\bx) + \alpha \eta df(\bx) \bv,
\end{eqnarray*}
for fixed $\alpha \in (0,1)$. We give a modest generalization of the proof given in Chapter 12 of \cite{Lan2013}, where we substitute Lipschitz differentiability for twice differentiability of $f(\bx)$.
\begin{proposition}\label{prop:backtrack}
Given $\alpha \in (0,1)$ and $\sigma \in (0,1)$, there exists an integer $s \geq 0$ such that
\begin{eqnarray*}
f(\bx + \sigma^s\bv) \leq f(\bx) + \alpha \sigma^s df(\bx) \bv,
\end{eqnarray*}
where $\bv = -\bH(\bx)^{-1} \nabla f(\bx)$ as in the update (\ref{eq:mm_update}) or update (\ref{eq:mm_update2}).
\end{proposition}

\begin{proof}
Let $\mathcal{L}_f(\bx_0) \equiv \{ \bx : f(\bx) \leq f(\bx_0)\}$. Since $f(\bx)$ is coercive, $\mathcal{L}_f(\bx_0)$ is compact. Because $\nabla f(\bx)$ and $\bH(\bx)$ are continuous, %where $\bH(\bx)$ denotes the second differential term, 
there are positive constants $a$ and $b,$ such that
\begin{eqnarray*}
a \leq \lVert \bH(\bx) \rVert \leq b,
\end{eqnarray*}
for all $\bx \in \mathcal{L}_f(\bx_0)$. Since $\nabla f(\bx)$ is $L$-Lipschitz continuous, 
\begin{eqnarray}
\label{eq:A}
f(\bx + \eta\bv) \leq f(\bx) + \eta df(\bx)\bv + \frac{1}{2} \eta^2L\lVert \bv\rVert^2.
\end{eqnarray}
The last term on the right hand side of (\ref{eq:A}) can be bounded by
\begin{eqnarray*}
\lVert \bv \rVert^2 = \lVert \bH(\bx)^{-1} \nabla f(\bx) \rVert^2 \leq a^{-2} \lVert \nabla f(\bx) \rVert^2.
\end{eqnarray*}
We next identify a bound for $\lVert \nabla f(\bx) \rVert^2$.
\begin{equation}
\label{eq:B}
\begin{split}
\lVert \nabla f(\bx) \rVert^2 & \amp = \amp \lVert \bH(\bx)^{1/2}\bH(\bx)^{-1/2} \nabla f(\bx) \rVert^2 \\
& \amp \leq \amp \lVert \bH(\bx)^{1/2}\rVert^2 \lVert\bH(\bx)^{-1/2} \nabla f(\bx) \rVert^2 \\
& \amp \leq \amp b df(\bx) \bH(\bx)^{-1} \nabla f(\bx) \\
& \amp = \amp -b df(\bx) \bv.
\end{split}
\end{equation}
Combining the inequalities in (\ref{eq:A}) and (\ref{eq:B}) yields
\begin{eqnarray*}
f(\bx + \eta \bv) \leq f(\bx) + \eta \left (1 - \frac{b  L}{2\gamma^2}t \right ) df(\bx) \bv.
\end{eqnarray*}
Therefore, the Armijo condition is met for any $s \geq s^\star$, where
\begin{eqnarray*}
s^\star = \frac{1}{\log \sigma} \log \left ( \frac{2(1-\alpha)a^2}{b L} \right ).
\end{eqnarray*}

\end{proof}

We now prove the convergence of our MM algorithm safeguarded by Armijo step-halving.

\begin{proposition}
\label{prop:convergence}
Consider the algorithm map
\begin{eqnarray*}
\psi(\bx) = \bx - \eta_{\bx}\bH(\bx)^{-1} \nabla f(\bx),
\end{eqnarray*}
where $\eta_{\bx}$ has been selected by backtracking to ensure Armijo's condition.
Then the limit points of the sequence $\bx_{k+1} = \psi(\bx_{k})$ are stationary points of $f(\bx)$. Moreover, this set of limit points is compact and connected.
\end{proposition}

\begin{proof}

Let $\bv = - \bH(\bx)^{-1} \nabla f(\bx)$ and $\bx_{k+1} = \bx_{k} + \sigma^{s_k} \bv_{k}$. Since $f(\bx)$ is continuous, it attains its infimum on $\mathcal{L}_f(\bx_0)$, and therefore the decreasing sequence $f(\bx_{k})$ is bounded below. Thus, $f(\bx_{k}) - f(\bx_{k+1})$ converges to 0. Let $s_k$ denote the number of backtracking steps taken at the $k$th iteration, and assume that the line search is terminated once the Armijo condition is met. As a consequence of backtracking,
\begin{eqnarray*}
f(\bx_{k}) - f(\bx_{k+1}) & \geq & - \alpha \sigma^{s_k} df(\bx_{k}) \bv_{k} \\
& = & \alpha \sigma^{s_k} df(\bx_{k}) \bH(\bx_{k})^{-1}\nabla f(\bx) \\
& \geq & \frac{\alpha\sigma^{s_k}}{\beta} \lVert \nabla f(\bx_{k}) \rVert^2 \\
& \geq & \frac{\alpha\sigma^{s^\star + 1}}{\beta} \lVert \nabla f(\bx_{k}) \rVert^2. \\
\end{eqnarray*}
The inequality above implies that $\lVert \nabla f(\bx_{k}) \rVert$ converges to 0. Therefore, all the limit points of the sequence $\bx_{k}$
are stationary points of $f(\bx)$.
Taking norms of the update produces % rule, we arrive at the inequality
\begin{eqnarray*}
\lVert \bx_{k+1} - \bx_{k} \rVert & = & \sigma^{s_k} \lVert \bH(\bx_{k})^{-1} \nabla f(\bx_{k}) \rVert \\
& \leq & \sigma^{s_k} a^{-1} \lVert \nabla f(\bx_{k}) \rVert.
\end{eqnarray*}
Thus, we have a bounded sequence $\bx_{k}$ with $\lVert \bx_{k+1} - \bx_{k} \rVert$ tending to 0. According to Propositions 12.4.2 and 12.4.3 in \cite{Lan2013}, the set of limit points is compact and connected.
\end{proof}

Note that in the linear case $h(\bx) = A \bx$, the gradient $\nabla f(\bx)$ is $L$-Lipschitz continuous with $L = \sum_i v_i + \rho(\bA^t \bA) \sum_j w_j$, where $\rho(\bA^t \bA)$ is the spectral radius of $\bA^t \bA$. Finally, note that we do not assume convexity of the proximity function. Whenever the proximity function is convex, stationary points are furthermore necessarily global minimizers, and therefore all limit points of the iterate sequence are global minimizers. 

Before leaving our discussion on convergence, we address the issue of the convergence of the accelerated version of our algorithm. In practice, we observe that the accelerated algorithm converges to valid solutions, but the fact that its iterates may temporarily increase our objective complicates convergence analysis. Nonetheless, convergence of the acceleration scheme can be guaranteed with the following slight modification. Each quasi-Newton step requires two ordinary MM iterations \cite{ZhoAleLan2011}. If the quasi-Newton update fails to reduce the objective function, then we revert to the second iterate $\psi \circ \psi(\bx_{k})$. This modified scheme comes with convergence guarantees described in the following proposition, proved in the Appendix.

\begin{proposition}
\label{prop:convergence_acc}
Let $\varphi$ denote the algorithm map corresponding to the modified quasi-Newton method. Generate the iterate sequence $\bx_{k+1} = \varphi(\bx_{k})$. 
Then the limit points of this sequence are stationary points of $f(\bx)$.
\end{proposition}

%% ----------------------------------------------------------------------
%% Application: Sparse Regression
%% ----------------------------------------------------------------------
\section{Application: Regression with Constraints}
Proximity function optimization offers a general framework encompassing statistical regression with constraints. For instance, we see that the two-set linear split feasibility problem with no domain constraint and a singleton range constraint $Q = \{ \by \}$ coincides with the Tikhonov regularization or ridge regression objective function $\lVert \bA \bx - \by \rVert_2^2 + \lambda \lVert \bx \rVert_2^2 $. This is an example of regularized linear regression, where the objective jointly minimizes the least squares objective along with a penalty function that encourages solutions to be well-behaved in some desired sense. We will see split feasibility provides a flexible approach toward this goal, and its non-linear and Bregman generalizations allow for extensions to generalized linear models and non-linear transformations of the mean relationship. For exposition, we focus on the important case of sparse regression.

%The theory of compressed sensing provides a framework for recovering an unknown signal from an undersampled set of measurements, given that the original signal is sparse. Often this signal can be recovered perfectly even when the measurement contains significantly fewer samples than dictated by the Shannon-Nyquist rate. 
In the high-dimensional setting where the dimension of covariates $n$ is larger than the number of observations $m$, it is often the case that we only expect few predictors to have nonzero impact on the response variable. The goal in sparse regression is to correctly identify the nonzero regression coefficients, % in the sparse support of $\bx$, 
and to estimate these coefficients accurately. 
Given an $m \times n$ matrix $\bA$ with $m \ll n$, and a vector $\by$ of (possibly noisy) observations, we seek to recover a sparse solution $\bx$ to the underdetermined system $\bA \bx = \by$. This can be formulated as solving the following problem
\begin{equation}\label{eq:cs_l0}
 \hat{\bx} = \argmin _{\bx} \lVert \bx \rVert_0 \hspace{12pt} s.t. \hspace{12pt} \lVert \bA \bx - \by \rVert_2 \leq \delta,
\end{equation}
where the parameter $\delta$ quantifies a variance level of the noise. 
We note that the same objective arises in compressed sensing, where $\bA$ represents the measurement and sensing protocol, and $\bx$ is a sparse signal or image to be recovered from undersampled measurements.

Because the non-convex $\ell_0$ norm renders this optimization problem combinatorially intractable, a convex relaxation that induces sparsity by proxy via the $\ell_1$ norm is often solved instead:
\begin{equation}\label{eq:cs_l1}
 \hat{\bx} = \argmin _{\bx} \lVert \bx \rVert_1 \hspace{12pt} s.t. \hspace{12pt} \lVert \bA \bx - \by \rVert_2 \leq \delta.
\end{equation}
The problem defined in (\ref{eq:cs_l1}), sometimes referred to as basic pursuit denoising, has an equivalent unconstrained formulation %The unconstrained formulation of this problem results in a convex optimization program:
\begin{equation}\label{eq:cs_noisy_l1}
 \hat{\bx}  = \argmin _{\bx} \frac{1}{2}\lVert \bA \bx - \by \rVert_2^2 + \lambda \lVert\bx\rVert_1, 
 \end{equation} where $\lambda >0 $ is a tuning parameter governing the amount of regularization. Although there is no interpretable relationship between $\lambda$ and the sparsity constant $k$, a suitable value for $\lambda$ is typically found using cross-validation. When $\bx$ is $k$-sparse so that $\lVert \bx \rVert_0 \leq k$ and $\bA$ satisfies certain incoherence properties (e.g.\@ Theorem 1.4 in \cite{Davenport2012}),  $m = \mathcal{O} (k \log n)$ measurements suffices to recover $\bx$ with high probability guarantees. 

Equation (\ref{eq:cs_noisy_l1}) is also known as the least absolute shrinkage and selection operator (LASSO), the most prevalent regularization tool for variable selection in statistics \cite{tibshirani1996}. %for $\ell_1$-penalized least squares regression, a prevalent technique for variable selection in statistics. 
In this setting, $\bA$ denotes the design matrix containing covariate information, $\bx$ the vector of regression coefficients, and $\by$ the vector of observations or response variables. 
It is well known that LASSO estimation simultaneously induces both sparsity and shrinkage. When the goal is to select a small number of relevant variables, or seek a sparse solution $\bx$, shrinking estimates toward zero is undesirable. This bias can be ameliorated in some settings using alternative penalties such as minimax concave penalty (MCP) and smoothly clipped absolute deviation (SCAD) \cite{fan2010}, or corrected by re-estimating under the model containing only those selected variables when the support is successfully recovered. However, the shrinkage induced by LASSO notoriously introduces false negatives, leading to the inclusion of many spurious covariates. 

Instead, we may directly incorporate the sparsity constraint using the split feasibility framework by defining a domain constraint as the sparsity set
\[ C = \left\{ \bz \in \Real^n : \lVert \bz \rVert_0 \leq k \right\}. \] 
Instead of seeking a solution to (\ref{eq:cs_l0}) by way of solving (\ref{eq:cs_l1}), we now attempt to recover the sparse solution $\bx$ by minimizing the proximity function (\ref{eq:surrogate}) with $h(\bx) = \bA \bx$, $C$ as above, and $Q = \{ \by \}$.
Projecting a point $\bz$ onto $C$ is easily accomplished by setting all but the $k$ largest (in magnitude) entries of $\bz$ equal to $0$; this is the same projection required within iterative hard thresholding algorithms. Note that $C$ is non-convex, 
%and global convergence guarantees from Section \ref{sec:converge} do not hold. Nevertheless, limit points of the algorithm are stationary points leading to local optima, and 
but we will see that in practice the MM algorithm nevertheless recovers good solutions using such projections. 

Unlike LASSO and related continuous penalty methods, this formulation enables us to directly leverage prior knowledge of the sparsity level $k$ and does not cause shrinkage toward the origin. When no such information is available, $k$ can be selected by analogous cross validation procedures used to tune the parameter $\lambda$. Split feasibility provides further advantages. It is straightforward to include multiple norm penalty terms, for instance akin to elastic net regularization to encourage grouped selection \cite{zou2005}, and additional generic set-based constraints. When an approximate solution suffices, relaxing the range constraint can easily be accomplished by replacing $Q = \{ \by \}$ by $Q = B_\varepsilon(\by)$ the $\varepsilon$-ball around $\by$. 
Further, because our MM algorithm is not confined to the linear case, smooth transformations $h(\bA \bx)$ of the mean relationship can immediately be considered. This is preferable when directly transforming the data is problematic: for example, when a log-linear relationship 
$ \log \by = \bA \bx $ %+ \boldsymbol\varepsilon\] 
is expected but values $y_i \leq 0$ are observed. %; here $\boldsymbol\varepsilon$ denotes Gaussian errors. 

\begin{figure}[htbp]
\centering
\includegraphics[scale=0.7]{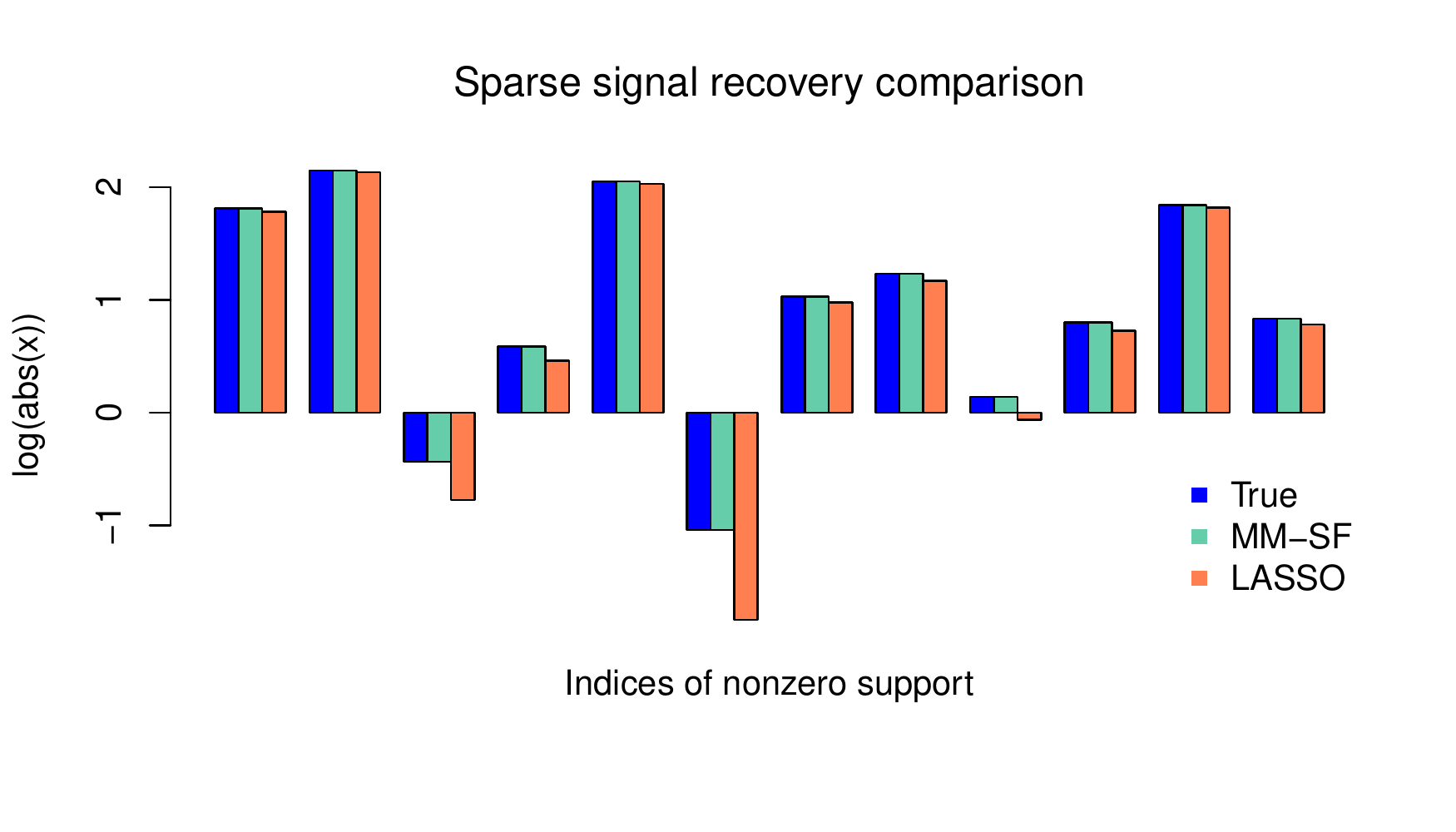} \vspace{-30pt}
\caption{Comparison of LASSO and MM approaches to noiseless recovery of a sparse signal on log scale for visual clarity. Both solutions correctly recover the nonzero support of $\bx$. The MM solution is virtually identical to the true vector, while the LASSO solution suffers visible shrinkage toward zero in each component.} 
\label{fig:l0_compare}
\end{figure}
\paragraph{\bf Performance:} To assess performance, we compare solutions to sparse regression problems using our MM algorithm and LASSO. %as implemented in the R package \texttt{glmnet}. 
To this end, we create a $300 \times 3000$ matrix $\bA$ with entries drawn from a standard normal distribution $N(0,1)$, and a true signal $\bx$ with $k=12$ nonzero entries, each drawn from $N(0,5)$ and randomly placed among the $3000$ components of $\bx$. We then generate a noiseless response vector $\by = \bA\bx$, and seek to recover $\bx$ from the observations $\by$. Figure~\ref{fig:l0_compare} displays the results of one trial of this experiment: we see that while both methods successfully recover the support of $\bx$, the MM solution under the split feasibility formulation does not suffer from the visible shrinkage that LASSO suffers. 
\begin{figure}[htbp]
\centering
\includegraphics[scale=0.7]{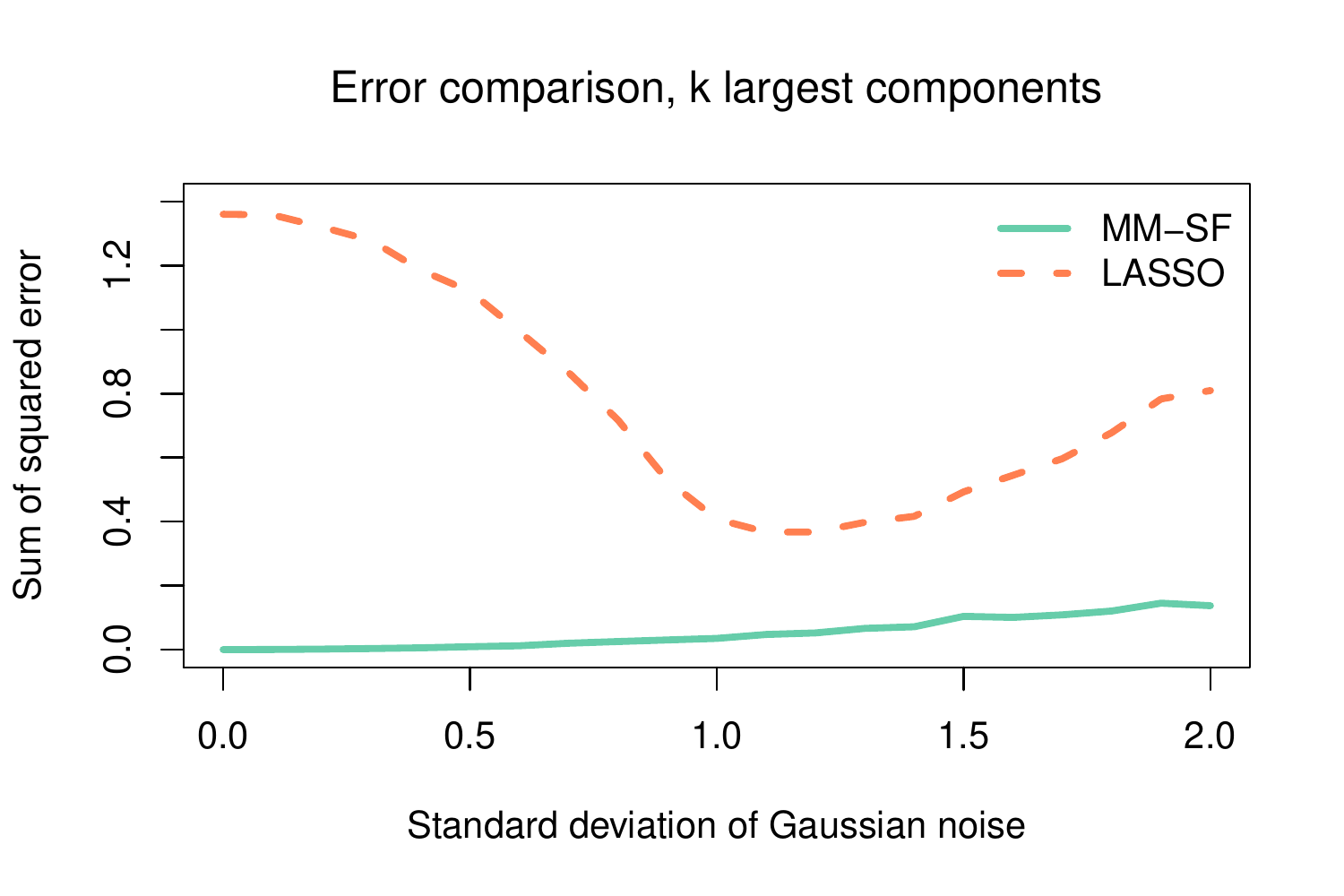} \vspace{-12pt}
\caption{Comparison of squared error among the true support of $\bx$ between solutions using LASSO and MM in the noisy case, averaged over $50$ trials. As LASSO does not benefit from prior knowledge of the sparsity level $k$, only the $k$ largest coefficients are considered; false positives are not penalized in assessing error above. The MM solution nonetheless features lower error for all noise levels $\sigma = 0, 0.1, \ldots, 2.0$. }
\label{fig:compare2}
\end{figure}

A comparison in the noisy case is presented in Figure \ref{fig:compare2}. Here $\bA$ is generated as before, and the true coefficient vector $\bx$ has $k=12$ nonzero entries each randomly assigned $5$ or $-5$ with probability $1/2$. Now, the response vector is given by $\by = \bA\bx + \varepsilon$, where $\varepsilon$ is distributed $N(0,\sigma^2)$ for a range of standard deviation values $\sigma = \{ 0, 0.1, 0.2, \ldots, 2.0 \}$. 
Results in Figure \ref{fig:compare2} report the total sum of squared error among the $k$ largest components of solution vectors obtained using LASSO and using the MM approach, averaged over $50$ trials of this experiment for each value of $\sigma$. Considering only the projection onto the $k$-sparsity set allows for a conservative comparison, as LASSO frequently includes false positives which are discarded by only assessing error on the true support. Again, the MM solution achieves lower squared error rate.

\paragraph{\bf Generalized linear models:} An extension of the above framework to generalized linear models (GLM) is made possible through Bregman divergences. The GLM is a flexible generalization of normal linear regression that allows each response $y$ to have non-Gaussian error distributions through specifying the mean relationship $ \mathbb{E}[y] = g^{-1}( \ba^T \bx )$ for a given link function $g$ \cite{mccullagh1989}. The response variable is assumed to be generated from an exponential family likelihood conditional on the covariate vector $\ba$:
\begin{equation}\label{eq:GLM}
p( y | \bx, \ba, \tau ) = C_1(y) \exp \left\{ \frac{ y \ba^T \bx - \psi( \ba^T \bx) }{C_2(\tau)} \right\}. 
\end{equation}
Here $\tau$ is a fixed scale parameter, $C_1, C_2$ are constant functions with respect to $\bx$, and $\ba$ should be thought of as a row of the design matrix $\bA$. The convex function $\psi$ is the log-partition or cumulant function. In this form, the canonical link function $g^{-1} = \psi'$.
For instance, when $\psi(x) = x^2/2$ and $C_2(\tau) = \tau^2$, (\ref{eq:GLM}) coincides with normal linear regression with identity link function. GLMs unify a range of regression settings including Poisson, logistic, gamma, and multinomial regression.
Note that a direct transform $h(\bA \bx)$ assumes a fundamentally different error structure than the GLM: as an example, a log-linear transform of the mean relationship posits that $\mathbb{E}[ \log (\by)] = \bA \bx$, while a GLM with log link function assumes $\log( \mathbb{E}[\by]) = \bA \bx$. 

We have seen that linear regression and its penalized modifications straightforwardly fall under the split feasibility framework. This is because maximizing the likelihood in the Gaussian case is equivalent to minimizing a least squares term, which appears in the proximity function (\ref{eq:proximity_function1}) as a Euclidean distance. While general exponential families do not share this feature, a close relationship to Bregman divergences allows an analogous treatment of GLMs. Specifically, the Legendre conjugate of $\psi$, which we denote $\zeta$, uniquely generates a Bregman divergence $D_\zeta$ that represents the exponential family likelihood up to proportionality \cite{polson2015}. That is, 
\[  - \log p( y | \bx, \ba, \tau) = D_\zeta \left( y , g^{-1} (\ba^T \bx ) \right) + C(y, \tau). \]
For instance, the cumulant function of the Poisson likelihood is $\psi(x) = e^x$, whose conjugate $\zeta(x)=x \log x - x$ produces the relative entropy  or Kullback-Leibler divergence $D_\zeta(x,y) = x \log (x/y) - x + y$.
We see that maximizing the likelihood in an exponential family is equivalent to minimizing a Bregman divergence, and thus generic constraints applied to the GLM such as sparsity sets can be considered as before by minimizing the proximity function %(\ref{eq:proximity_function2}) in the form 
\begin{eqnarray}
\label{eq:prox_glm}
f(\bx) & = & \sum_i v_i D_\phi \left (\mathcal{P}^\phi_{C_i}(\bx) , \bx \right) + 
\sum_{j=1}^m  D_\zeta \left ( y_j , g^{-1}( \ba_j^T\bx) \right).
\end{eqnarray} 

%% ----------------------------------------------------------------------
%% Application: Non-linear Complementarity
%% ----------------------------------------------------------------------
\subsection{Non-linear Complementarity and Quadratic Programming}
In the non-linear complementarity problem we are given a smooth function 
$u(\bx): \Real^p \mapsto \Real^p$, and seek an $\bx \ge {\bf 0}$
such that $u(\bx) \ge {\bf 0}$ and $u(\bx)^t \bx = 0$ \cite{RocWet1998}.  This becomes a subproblem within our current framework by taking% $h(\bx) = \begin{pmatrix} \bx \\ u(\bx) \end{pmatrix}$, and
%$C = \{\bx \in \Real^p: \bx \ge {\bf 0}\}$, and
\begin{eqnarray*}
h(\bx) & = & \begin{pmatrix} \bx \\ u(\bx) \end{pmatrix}\quad \text{and}\quad 
D = \left \{\begin{pmatrix} \bx \\ \by \end{pmatrix} \in \Real^{2p}: \bx \ge {\bf 0}, \:\by \ge {\bf 0}, \: \bx^t\by = 0 \right\}.
\end{eqnarray*}
The fact that $D$ is closed but not convex does not prevent us from straightforwardly
projecting onto it. %We just have to contend with multivalued projections.
The set $D$ splits into the Cartesian product of the
sets $D_i  = \{(x_i,y_i): x_i \ge 0, y_i \ge 0,\, x_iy_i=0\}$. Therefore the projection
operator $\mathcal{P}_D(\bu,\bv)$ is the Cartesian product of the
projection operators $\mathcal{P}_{D_i}(u_i,v_i)$, defined by
\begin{eqnarray*}
\mathcal{P}_{D_i}(u_i,v_i) & = & \begin{cases} (u_i,0) & u_i>v_i \ge 0 \\
(0,v_i) & v_i>u_i \ge 0 \\ (u_i,0) \: \text{or} \: (0,v_i) & u_i=v_i \ge 0  \\
(\max\{u_i,0\},\max\{v_i,0\}) & \text{otherwise .} \label{lcp_project}
\end{cases}
\end{eqnarray*}
The relevant gradient is now $\nabla h(\bx) = [\bI_{n},\nabla u(\bx)]$. As before,
second derivatives are unnecessary in finding a split feasible point. 

Notice the well-studied linear complementarity problem is the special case with $u(\bx) = \bM \bx + \bq$. Namely, given a real matrix $\bM$ and vector $\bq$, we seek a point $\bx $ such that $\bx \ge {\bf 0}, \bM \bx + \bq \geq {\bf 0},$ and $\bx^t (\bM \bx + \bq ) = 0$. This can be recast as minimizing the quantity $\bx^t (\bM \bx + \bq ) $ subject to the other non-negativity constraints. The minimum is $0$ if and only if the optimal solution solves the linear complementarity problem.
As a simple concrete example, consider the convex quadratic program stated with the goal of minimizing
$\bc^t \bx + \frac{1}{2} \bx^t \bB \bx$ subject to $\bA \bx \geq \bb, \bx \geq {\bf 0}$, where $\bB$ is a symmetric matrix. This problem reduces to a complementarity problem via the Karush-Kuhn-Tucker conditions: it is straightforward to show that the equivalent linear complementarity problem is defined by
$\bq = \begin{pmatrix} \bc \\ -\bb \end{pmatrix}$, $\bM = \begin{pmatrix} \bB & -\bA^t \\ \bA & {\bf 0} \end{pmatrix}$. Thus, the split feasibility framework yields a solution technique for quadratic programming. Here $h(\bx) = \begin{pmatrix} \bx \\ \bM \bx + \bq \end{pmatrix}$, and the update (\ref{eq:mm_update}) has closed form since it involves minimizing the squared distance between $\bx_{k}$ and its projection. If we denote the projection given by (\ref{lcp_project})
$\mathcal{P}_D[ h(\bx_{k})] = \begin{pmatrix} \ba_n \\ \bb_n \end{pmatrix}$, then
$\bx_{k+1} = \bG \begin{pmatrix} \ba_n \\ \bb_n - \bq \end{pmatrix}$ where the pseudoinverse $\bG = (\bM^t \bM)^{-1} \bM^t$ can be precomputed once and stored for use at each iteration.
The resulting algorithm is elegantly transparent and very straightforward to implement: at each step, project $\bx_{k}$ onto $D$, and then update $\bx_{k+1}$ via one matrix-by-vector multiplication. A na\"{\i}ve Julia implementation of this procedure produces solutions consistent with highly optimized solvers in the commercial optimization package Gurobi. Though not as computationally efficient, our algorithm is nonetheless appealing in its simplicity, and furthermore provides a notion of best solution when faced with an infeasible quadratic program. Such approximate solutions when not all constraints can be satisfied are valuable in many settings, as we will see in the following case study.

%% ----------------------------------------------------------------------
%% IMRT Problem
%% ----------------------------------------------------------------------
\section{Case study: Intensity-modulated Radiation Therapy}
\label{sec:imrt}
We now turn to a second application of split feasibility toward optimization for intensity-modulated radiation therapy (IMRT) data \cite{PalMac2003}. 
In IMRT, beams of penetrating radiation are directed at a tumor from external sources. We refer the interested reader to the review papers \cite{EhrGHam2010,SheFerOli1999} for more details on this rich problem.  Treatment planning consists of three optimization problems: beam-angle selection, fluence map optimization, and leaf sequencing. The linear split feasibility framework has played a significant role toward solving the second problem, which we focus our attention on here \cite{CenBorMar2006}. 

During treatment planning, a patient's tissue is partitioned into $m$ voxels which collectively comprise $p$ regions or collections of indices $\mathcal{R}_j$, so that $\{1, \ldots, m\}$ is the disjoint union $\cup_{j=1}^p \mathcal{R}_j$. The $j$th region $\mathcal{R}_j$ is denoted as a target, critical structure, or normal tissue. 
The dose of radiation that voxel $i$ receives by a unit weight, or intensity, in beamlet $l$ is denoted by the $il$th entry in the dose matrix $\bA \in \Real^{m \times n}$. Since doses scale linearly with weight and are additive across beamlets, we can express concisely the dosage $\bd = \bA\bx$ where $\bx \in \Real^n$ is a vector of beamlet weights. 

Target voxels are those containing tumorous tissue and surrounding tissue margins. The goal is thus to identify a beam weight assignment $\bx$ such that the target voxels receive close to a specified dose of radiation, while the critical and normal regions do not receive a dose in excess of a specified tolerance. Critical structures are distinguished from normal tissue, since failing to control the radiation exposure to the former can lead to serious effects on the patient's quality of life. Damage to normal tissue, on the other hand, is undesirable but more tolerable, and exposing normal and critical tissue to some level of radiation is unavoidable. Target regions may abut critical structures, or there may simply be no path to the target that is free of non-tumorous tissue. For modeling purposes, we will lump together the critical and normal voxels, referring to them collectively as non-target regions. Our formulation will distinguish the two using the weights $w_j$.

These problems typically involve 1,000-5,000 beamlets and up to 100,000 voxels \cite{AlbMeeNus2002,HouWanChe2003,LlaDeaBor2003,ZhaLiuWan2004}. This results in a linear split feasibility problem with a large number of equations and constraints. We will see that casting the problem non-linearly can drastically reduce the size of the system of  equations to be solved at every iteration. Begin by arranging the rows of $\bA$ so that voxels corresponding to targets and non-target structures are in contiguous blocks. So, we can write $\bA = [\bA_1^t \cdots \bA_p^t]^t$.
Suppose the $j$th region $\mathcal{R}_j$ is a non-target, so that we wish to limit the dose in each voxel of $\mathcal{R}_j$ to be at some specified maximum dose $d_j$. We can write this na\"{\i}vely as
\begin{eqnarray*}
\sum_{l=1}^n a_{il} x_l \leq d_j,
\end{eqnarray*}
for all $i \in \mathcal{R}_j$, translating the fluence map optimization to a linear feasibility problem with mapping $h(\bx) = \bA\bx$ and simple voxel-by-voxel box constraints. This is straightforward but leads to a large system of equations to solve, even after employing the Woodbury formula. Alternatively, note that this constraint can be written more concisely as % we can express this constraint more concisely as
\begin{eqnarray*}
\max (\bA_j \bx) \leq d_j
\end{eqnarray*}
Since our MM algorithm requires differentiable mappings, we replace the $\max$ appearing above by the softmax mapping $\mu : \Real^n \rightarrow \Real$ given by
\begin{eqnarray*}
\mu(\bx) & = & \frac{1}{\gamma} \log \left (
\sum_{l=1}^n \exp(\gamma x_l)
\right ),
\end{eqnarray*}
with parameter $\gamma > 0$. The softmax is always greater than the $\max$, and the approximation improves as $\gamma \rightarrow \infty$. Likewise, the softmin can be calculated by $-\mu(-\bx)$ and is always less than $\min(\bx)$. Now, let $h_j(\bx) = \mu(\bA_j \bx)$ when the $j$th region is a non-target, and let $h_j(\bx) = - \mu(-\bA_j \bx)$ when the $j$th region is a target. Then 
$h : \Real^n \rightarrow \Real^p$ leads to a non-linear ``region-by-region" formulation that may significantly reduce the size of the optimization problem, as typically the number of regions $p$ will be dramatically smaller than the number of voxels $m$. If the $j$th region
is a non-target, then the range constraint on the region $\mathcal{R}_j$ is $Q_j = \{ \by \in \Real^p : y_i \leq d_j, \forall i \in \mathcal{R}_j \}$.
Alternatively, if the $j$th region
is a target, then the range constraint on the region $\mathcal{R}_j$ is $Q_j = \{ \by \in \Real^p : y_i \geq d_j, \forall i \in \mathcal{R}_j \}$.
Let $\mathcal{T} \subset \{1, \ldots, p\}$ denote the set of regions 
that are targets. Then
\begin{equation}
\label{eq:rr}
\begin{split}
f(\bx) & = \frac{1}{2} \sum_i v_i \dist(\bx,C_i)^2+
\frac{1}{2} \sum_j w_j \dist[h(\bx),Q_j]^2 , \nonumber \\ %\label{proximity1} \\
& = \frac{1}{2} \sum_i v_i \dist(\bx, C_i)^2 \nonumber \\ 
+ &
\frac{1}{2} \left [
\sum_{j \in \mathcal{T}^c} w_j [\mu(\bA_j \bx) - d_j]_+^2 + 
\sum_{j \in \mathcal{T}} w_j [d_j + \mu(-\bA_j \bx)]_+^2
\right ],
\end{split}
\end{equation}
where $[u]_+ = \max(u,0)$. Under the non-linear mapping $h(\bx)$ and constraint sets $Q_j$, the proximity function (\ref{eq:rr}) is convex. 
Since $\exp(\gamma x_i)$ is log-convex, $\mu(\by)$ is convex, and as the composition of a convex function with an affine mapping preserves convexity, $\mu(\bA\bx)$ is convex. Finally, since $[u]_+$ is a non-decreasing convex function, the composition $[\mu(\bA_j \bx) - d_j]_+$ is convex, and therefore its square is also convex. Thus, by Proposition~\ref{prop:convergence}, the limit points of the iterate sequence are guaranteed to be global minimizers of the proximity function.

The necessary Jacobian $\nabla h(\bx)$ appearing in the MM update is available in closed form based on the partial derivative
\begin{eqnarray*}
\frac{\partial}{\partial x_i} \mu(\bx) & = & \frac{ \exp(\gamma x_i) }{\sum_{j=1}^n \exp(\gamma x_j)}. 
\end{eqnarray*}
In vector notation, this amounts to
\begin{eqnarray*}
\nabla \mu(\bx) & = & \frac{1}{\sum_{j=1}^n \exp(\gamma x_j)} \exp(\gamma \bx),
\end{eqnarray*}
where $\exp(\gamma \bx) = \begin{bmatrix}\exp(\gamma x_1) & \cdots & \exp(\gamma x_n)\end{bmatrix}^t$. The chain rule gives us 
%$\nabla h_j(\bx) = \bA_j^t \nabla \mu(\bA_j \bx)$, and therefore
\begin{eqnarray*}
\nabla h(\bx) & = & \begin{bmatrix} \bA_1^t \nabla \mu(\bA_1 \bx) & \cdots & \bA_p^t \nabla \mu(\bA_r \bx) \end{bmatrix}. \\
\end{eqnarray*}

{\bf Results and Performance:} We apply our MM algorithm to seek optimal beam weight assignments in publicly available liver and prostate CORT datasets \cite{Craft2014}. The dose influence matrix features $m=47,089$ rows (voxels) in the liver dataset, corresponding to a $217 \times 217$ computerized tomography (CT) cross-section of the liver, and $n=458$ columns whose indices correspond to beamlets. Liver tissue is partitioned into two target and two non-target regions. The prostate data consists of two target regions and five non-target regions, with $m=33,856$ and $n=721$. In addition to lower and upper dosage constraints on target and non-target voxels respectively, we impose a positivity constraint on the domain of beamlet weights.

\begin{figure}
\centering
\includegraphics[scale=0.55]{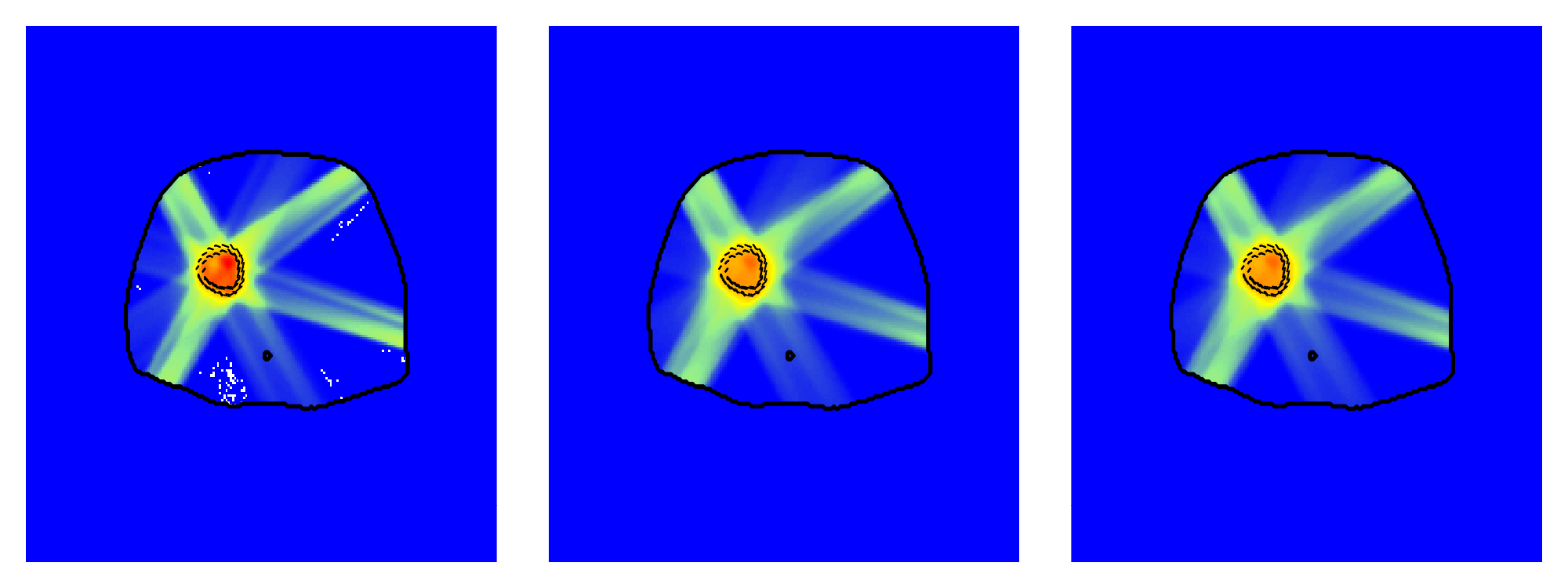} \includegraphics[scale=0.55]{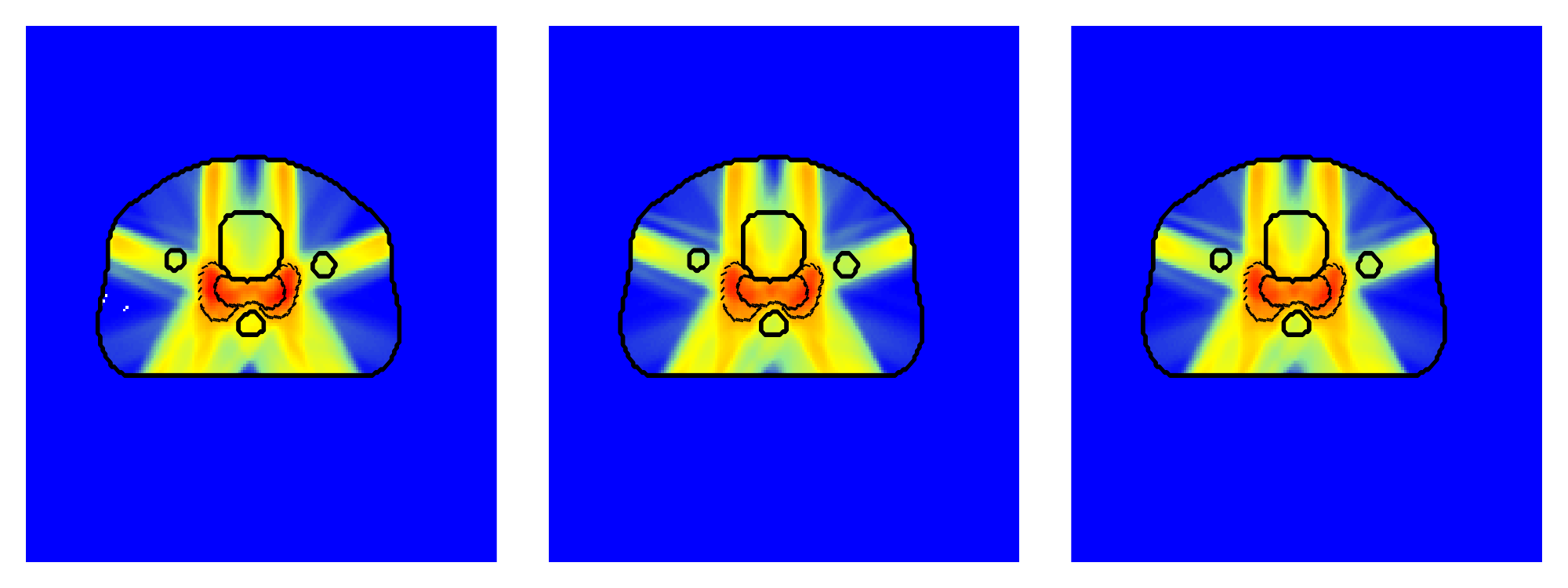} 
\caption{From left to right, solutions to the original linear (voxel-by-voxel) feasibility problem, the non-linear (region-by-region) version, and the region-by-region problem under the beta divergence with $\beta=4$ on cross-sectional CT liver data (top panel) and prostate data (bottom panel). Target regions are enclosed by dashed lines, while non-target regions are outlined with solid lines. Blue pixels receive no radiation, while warmer colors indicate higher doses. Negative entries of solutions appear as white pixels. }
\label{fig:IMRT_liver}
\end{figure}

For each dataset, we solve the linear split feasibility problem as well its non-linear region-by-region approximation described above. The linear problem is solved na\"{\i}vely via the update (\ref{eq:mm_update}), as well as using an efficient direct solve, replacing (\ref{eq:mm_update}) in the algorithm by the exact update 
\[ \bx_{k+1}  =  (v\bI + w\bA^t\bA)^{-1} \left[ \sum_i \bv_i \mathcal{P}_{C_i}(\bx_{k}) + \bA^t \sum_j w_j \bA \mathcal{P}_{Q_j}(\bA\bx_{k}) \right]. \] 
Matrix inverses are precomputed via Cholesky decomposition and cached in the linear case. Furthermore, we consider a Bregman generalization of the region-by-region formulation using the $\beta$-divergence in place of Euclidean projections. Details, derivations, and complexity analysis are provided in the Appendix. A visual comparison of solutions in each case is displayed in Figure \ref{fig:IMRT_liver}. Initializing each trial from $\bx_0 = \mathbf{0}$, all three solutions are qualitatively similar in both datasets, illustrating that the smooth relaxation provides a good approximation to the original IMRT problem. Target regions receive higher dosage as one would expect, although solutions to the non-linear problems appear to benefit from smoothing. Solutions to the linear voxel-by-voxel problem contain scattered negative entries, an undesirable effect that is more apparent in the liver dataset.

\begin{table}[htbp]
\centering
\begin{tabular}{| l |  c c c c | } 
 \multicolumn{5}{c}{Median (MAD), liver data}  \\ \hline
Model & Linear & Linear, Direct & Non-linear & Bregman  \\ \hline
Time (sec) & 988.9 (193.7) & 114.9 (20.5) &  16.8 (3.9) & 10.3 (2.9) \\
Ref. Obj. & $7.17 \mathrm{e}$-3 ($7.12 \mathrm{e}$-4) &  $7.87 \mathrm{e}$-3 ($6.54 \mathrm{e}$-4) &  $2.35\mathrm{e}$-2 ($5.01\mathrm{e}$-3) & $8.81 \mathrm{e}$-3 ($2.24 \mathrm{e}$-3) \\ \hline
\end{tabular}
\begin{tabular}{| l | c c c c | }
 \multicolumn{5}{c}{Median (MAD), prostate data}  \\ \hline
Model & Linear & Linear, Direct & Non-linear & Bregman \\ \hline
Time (sec) & 1842.1 (357.8) & 138.1 (20.6) & 22.9 (3.6) & 12.9 (1.6) \\
Ref. Obj. & $1.42\mathrm{e}$-2 ($6.54\mathrm{e}$-4) & $1.50\mathrm{e}$-2 ($4.95\mathrm{e}$-4) & $6.17\mathrm{e}$-2 ($9.31\mathrm{e}$-3) & $1.73 \mathrm{e}$-2 ($2.15\mathrm{e}$-3) \\ \hline
\end{tabular}

\caption{Comparison of runtimes and voxel-by-voxel objective values in the linear, linear with direct solve, non-linear, and Bregman formulations of the split feasibility problem applied to liver and prostate IMRT data. Medians and median absolute deviations (MAD) are reported over $25$ uniform random initializations. }
\label{tab:IMRT}
\end{table}

A more detailed performance comparison is included in Table \ref{tab:IMRT}. In each case, results are reported over $25$ random initializations, with entries of $\bx_0$ drawn uniformly between $0$ and $10$, and relative tolerance set at $1\mathrm{e}$-6 to determine convergence. The Bregman column corresponds to results employing the $\beta$-divergence with $\beta=4$ in place of the Euclidean norm. Because each formulation optimizes a distinct objective function, all solutions are inserted into the voxel-by-voxel objective function, which serves as a reference objective enabling direct comparison. 

From Table \ref{tab:IMRT}, it is evident that the linear feasibility problem requires an order of magnitude more computing time than the non-linear and Bregman versions, despite producing similar solutions. This remains true despite caching the matrix inverse for repeated use in each iteration of the linear problem, and directly solving for the exact minimum in each iteration in place of the update given by (\ref{eq:mm_update}) used in all other versions. We remark that the number of columns (beamlets) in the IMRT datasets is less than a thousand. In applications with very large matrices, caching in the linear case becomes of limited use when even one matrix inversion becomes computationally prohibitive. On the other hand, a non-linear formulation may drastically reduce the inherent size of the optimization problem, and together with the Woodbury formula, enables our algorithm to be applied in such cases. 

In terms of the reference objective, all formulations achieve a low value, but it is notable that solutions under $\beta$-divergence achieve an objective value almost as low as the linear case, without directly optimizing the voxel-by-voxel loss. In addition, Figure \ref{fig:IMRT_liver} suggests that solutions to the region-by-region problems do not violate the positivity constraint on the domain. While some degree of range constraint (dose limit) violation is typically inevitable in this application, the negative weights appearing in the solution to the linear case lack interpretability to practitioners, and are therefore more problematic.  These empirical results illustrate several clear advantages in pursuing non-linear and Bregman formulations of split feasibility problems.

%% ----------------------------------------------------------------------
%% Discussion
%% ----------------------------------------------------------------------
\section{Discussion}

This paper extends the proximity function approach used to solve the classical linear split feasibility function. Using the versatile MM principle, we can now accommodate non-linear mappings and general Bregman divergences replacing the Euclidean distance. Our algorithm does not require second derivative information and is amenable to acceleration, for instance via quasi-Newton schemes relying on simple secant approximations. Under weak assumptions, we prove that limit points of the algorithm are necessarily stationary points of the proximity function. Thus, convergence to a global minimum is guaranteed for convex sets. A variety of examples that fall under the non-linear split feasibility framework are included, some of which involve non-convex constraint sets. Nonetheless, our algorithm is capable of finding valid solutions in such settings, demonstrating that its practical use extends beyond nicely behaved settings where iterates are guaranteed to converge to a global minimizer.

A variety of statistical applications and inverse problems arising in scientific disciplines can be studied within the split feasibility frameworks. The generalization to non-linear maps and arbitrary Bregman divergences further widens the range of applications, as illustrated with an application to sparse regression with set-based constraints for generalized linear models. We focus on an application to optimization for intensity-modulated radiation therapy data, an historical application of linear split feasibility. The MM algorithms presented here consistently outpace classical linear algorithms on two datasets while producing comparable, and arguably better, solutions. 

\begin{acknowledgements}
We thank Steve Wright and D\'{a}vid Papp for their help with the IMRT data examples.
\end{acknowledgements}

%% ----------------------------------------------------------------------
%% Appendix
%% ----------------------------------------------------------------------
\section*{Appendix}

%To include:
%\begin{itemize}
%\item QP comparison vs Gurobi
%\end{itemize}
%% ----------------------------------------------------------------------
%% Proof of monotone quasi-Newton version
%% ----------------------------------------------------------------------
\subsection{Proof of Proposition~\ref{prop:convergence_acc}}

Our proof of Proposition~\ref{prop:convergence_acc} relies on showing the algorithm map $\psi$ that carries out the Armijo backtracking line search is closed at all non-stationary points of $f$. Recall that a point-to-set map $A$ is closed at a point $\bx$, if for any sequence $\bx_j \rightarrow \bx$ such that $\by_j \in A(\bx_j)$ converges to $\by$, it follows that $\by \in A(\bx)$. Define the following point-to-set map
\[ S(\bx, \bv) = \left\{ \by : \by = \bx + \sigma^k \bv \text{ for $k \in \mathbb{Z}_+$ such that}\, f(\by) \leq f(\bx) + \alpha \sigma^{k} df(\bx ) \bv \right\},\]
and let $G$ denote the map
\begin{eqnarray*}
G(\bx) = \begin{pmatrix}
\bx \\
\nabla f(\bx)
\end{pmatrix}.
\end{eqnarray*}
The map $G$ is continuous since $f$ is Lipschitz differentiable. Let $A = S \circ G(\bx)$. By Corollary 2 in Chapter 7.7 of \cite{Luenberger2016}, the composite mapping $A$ is closed at $\bx$, if the mapping $S$ is closed at $G(\bx)$. We will state and prove a slightly more general result.

\begin{proposition}
\label{prop:closed}
The mapping $S$ is closed at all points $(\bx, \bv)$ provided that $\bv \not = \bZero$.
\end{proposition}

\begin{proof}
Consider sequences $\{ \bx_j \}$, $\{ \bv_j \}$ such that $\bx_j \rightarrow \bx, \bv_j \rightarrow \bv$, and let $\by_k \in S(\bx_{k}, \bv_k)$ with $\by_k \rightarrow \by$. For every $j$, there is some $k_j$ such that $\by_j = \bx_j + \sigma^{k_j} \bv_j$; rearranging we have that \[\sigma^{k_j} = \log \left (\frac{ \lVert \by_j - \bx_j \rVert }{\lVert \bv_j \rVert} \right). \] Taking limits on both side above yields 
\[ \sigma^{\bar{k}} = \frac{ \lVert \by - \bx \rVert }{\lVert \bv \rVert}, \text{ where }\bar{k} = \lim_{j \rightarrow \infty} k_j. \] Because $\by_k \in S(\bx_{k}, \bv_k)$, it holds for each $j$ that 
\begin{equation}\label{eq:ylim}
 f(\by_j) \leq f(\bx_j) + \alpha \sigma^{k_j} df(\bx_j ) \bv_j . \end{equation}
Since $\{k_j\}$ is a sequence of integers, it assumes only finitely many distinct values before converging to the constant sequence $\{ \bar{k} \}$; let $k^\dagger$ denote the maximum of these values. Then replacing $k_j$ by $k^\dagger$ and letting $j \rightarrow \infty$ in (\ref{eq:ylim}), together with continuity of $f$ and $df(\bx)$, we have that 
\[  f(\by) \leq f(\bx) + \alpha \sigma^{k^\dagger} df(\bx ) \bv . \]
That is, $\by \in S(\bx, \bv)$, proving that $S$ is closed at $(\bx, \bv)$.
\end{proof}

Proposition~\ref{prop:closed} informs us that the map $A$ is closed at all non-stationary points of $f$. We are now ready to prove Proposition~\ref{prop:convergence_acc}.

\begin{proof}
Fix an initial iterate $\bx_0$. Note that the set $\mathcal{L}_f(\bx_0) \equiv \{ \bx : f(\bx) \leq f(\bx_0)\}$ is compact since $f$ is coercive and the modified quasi-Newton method generates monotonically decreasing values. Since $\mathcal{L}_f(\bx_0)$ is compact, the sequence $\{\bx_{k}\}$ has a convergent subsequence whose limit is in $\mathcal{L}_f(\bx_0)$; denote this as
$\bx_{k_l} \rightarrow \bx_\star$ as $l \rightarrow \infty$. Our goal is to show that $\bx_\star$ must be a stationary point of $f$. To the contrary, suppose that $\bx_\star$ is not a stationary point of $f$.

Let $\by_{k_l} = \psi(\bx_{k_l}) \in A(\bx_{k_l})$. Note that $\by_{k_l} \in \mathcal{L}_f(\bx_0)$ and therefore the sequence $\{\by_{k_l}\}$
has a convergent subsequence $\{\by_{k_{l_j}}\}$. Denote this sequence's limit $\by_\star$. Note that the map $A$ is closed at $\bx_\star$, since $\bx_\star$ is not a stationary point of $f$. Therefore, by the definition of closed maps, we have that $\by_\star \in A(\bx_\star)$ and consequently
\begin{eqnarray}
\label{eq:algo_map}
f(\by_\star) \leq f(\bx_\star) + \alpha \sigma^k df(\bx_\star) \bH(\bx_\star)^{-1}\nabla f(\bx_\star) < f(\bx_\star),
\end{eqnarray}
for some positive integer $k$. On the other hand, since $f$ is Lipschitz-differentiable, it is continuous; therefore $\lim f(\bx_{k_l}) = f(\bx_\star)$. Moreover, for all $k_l$ we have that
\begin{eqnarray*}
f(\bx_\star) \leq f(\bx_{k_{l+1}}) \leq f(\bx_{k_l + 1}) = f(\varphi(\bx_{k_l})) \leq f(\psi \circ \psi(\bx_{k_l})) \leq f(\psi(\bx_{k_l})). \label{eq:ineq}
\end{eqnarray*}
By continuity of $f$, we have that $f(\bx_\star) \leq f(\by_\star)$, contradicting the inequality established in (\ref{eq:algo_map}). We conclude that $\bx_\star$ must be a stationary point of $f$.
\end{proof}

\subsection{$\beta$-divergence}
The $\beta$-divergence is defined
\begin{equation*}
D(\bx,\by) = \begin{cases} \displaystyle
\sum_{j}\dfrac{1}{\beta(\beta-1)} x_j^\beta + \sum_{j}\dfrac{1}{\beta} y_j^{\beta} - \dfrac{1}{\beta-1} x_j y_j^{\beta-1}  & \beta \in \Real \setminus \{0, 1 \}  \\
\sum_j x_j \log(\dfrac{x_j}{y_j}) - x_j + y_j & \beta = 1 \\
\sum_j \dfrac{x_j}{y_j} - \log(\dfrac{x_j}{y_j}) - 1 & \beta = 0 
\end{cases}
\end{equation*}
We see that the $\beta$-divergence corresponds to the Kullback-Leibler and Itakura-Saito divergences when $\beta=1,0$ respectively. Below we discuss the projection onto a hyperplane for the case of $\beta \in \Real \setminus \{0, 1 \}$.

%\noindent The first and second derivative with respect to y are continuous in $\beta$. The second derivative shows that the $\beta$-divergence is convex with respect to y for $\beta \in [1,2]$. Outside this interval, the divergence can always be expressed as the sum of a convex, concave and constant part, such that
%\begin{equation*}
%d_\beta(x|y) = \stackrel{\smallsmile}{d}(x|y) + \stackrel{\smallfrown}{d}(x|y) +\stackrel{-}{d}(x)
%\end{equation*}
The function $\phi(x) = \dfrac{1}{\beta(\beta-1)} \bx^\beta$ generates the $\beta$-divergence
\eq{D(\bx,\by) = \sum_{j}\dfrac{1}{\beta(\beta-1)} x_j^\beta + \sum_{j}\dfrac{1}{\beta} y_j^{\beta} - \dfrac{1}{\beta-1} x_j y_j^{\beta-1} .}
Its gradient is 
\eq{\nabla_\beta \phi(\bx) = \dfrac{1}{\beta-1}\bx^{\beta-1}, }
and recall the Fenchel conjugate of $\phi$ is given by
\eq{ \phi^\star(\bz) = \sup_{\bx}  \bigg( \langle \bz, \bx \rangle - \phi(\bx) \bigg) = \sup_{\bx} \bigg( \langle \bz, \bx \rangle - \sum_j \dfrac{1}{\beta(\beta-1)} x_j^\beta \bigg).}
\noindent Defining $h(\bx) =  \langle \bz, \bx \rangle - \sum_j \dfrac{1}{\beta(\beta-1)} x_j^\beta$, and differentiating with respect to $x_j$:
\eq{\nabla_{x_j} h &= z_j - \dfrac{1}{\beta-1}x_j^{\beta-1}  = 0 \\
x_j & = (\beta-1)^{\frac{1}{\beta-1}} z_j^{\frac{1}{\beta-1}} .}
\noindent Plugging into $\phi^\star(\bx)$,
\eq{\phi^\star(\bz) &= \sum_j z_j (\beta-1)^{\frac{1}{\beta-1}} z_j^{\frac{1}{\beta-1}} - \sum_j \dfrac{1}{\beta(\beta-1)} ( (\beta-1)^{\frac{1}{\beta-1}} z_j^{\frac{1}{\beta-1}})^\beta  \\
& =(\beta-1)^{\frac{1}{\beta-1}}  \sum_j z_j^{\frac{\beta}{\beta-1}} - \dfrac{ (\beta-1)^{\frac{1}{\beta-1}} }{\beta} \sum_j  z_j^{\frac{\beta}{\beta-1}} \\
& =(\beta-1)^{\frac{1}{\beta-1}} \bigg( 1  - \dfrac{1 }{\beta} \bigg) \sum_j  z_j^{\frac{\beta}{\beta-1}} .}
\noindent Finally, differentiating the Fenchel conjugate yields
\eq{\nabla \phi^\star(\bz) =(\beta-1)^{\frac{1}{\beta-1}} \bigg( 1  - \dfrac{1 }{\beta} \bigg)  \frac{\beta}{\beta-1}  \bz ^{\frac{1}{\beta-1}} =  (\beta-1)^{\frac{1}{\beta-1}}  \bz ^{\frac{1}{\beta-1} } .}
Thus, the projection of $\bx$ onto a hyperplane is given by
\eq{\mathcal{P}^\phi_{Q(a,c)} (\bx) =(\beta-1)^{\frac{1}{\beta-1}}  \bigg( \dfrac{1}{\beta-1} \bx^{\beta-1} - \widetilde{\gamma} a\bigg) ^{\frac{1}{\beta-1}},}
where $\widetilde{\gamma} = \text{arg}\min_{\gamma}
(\beta-1)^{\frac{1}{\beta-1}} \bigg( 1  - \dfrac{1 }{\beta} \bigg) \sum_j  \bigg( \dfrac{1}{\beta-1} x_j^{\beta-1} - \gamma a_j\bigg)^{\frac{\beta}{\beta-1}} + \gamma c .$

%% ----------------------------------------------------------------------
%% Per iteration complexity
%% ----------------------------------------------------------------------
\subsection{Per-iteration complexity of IMRT study}
We detail the calculations behind our per iteration complexity remarks.
Note that in the IMRT dataset considered, $p \ll n < m$.

\subsubsection{Linear}
How many flops are required to compute a single MM update given by Equation (\ref{eq:mm_update}) in the linear case?
\begin{eqnarray*}
\nabla f(\bx_{k}) & = & \sum_i \bv_i(\bx_{k} - \mathcal{P}_{C_i}(\bx_{k}))
+ \sum_j w_j \bA^t(\bA\bx_{k} - \mathcal{P}_{Q_j}(\bA\bx_{k})) \\
\bx_{k+1} & = & \bx_{k} - [v\bI + w\bA^t\bA]^{-1}\nabla f(\bx_{k}).
\end{eqnarray*}

We first tally the flops required to compute the gradient $\nabla f(\bx_{k})$.
In the IMRT example, the first sum $\sum_i v_i(\bx_{k} - \mathcal{P}_{C_i}(\bx_{k}))$ requires $\mathcal{O}(n)$ flops. The matrix-vector product $\bz_k = \bA\bx_{k}$ requires $\mathcal{O}(mn)$ flops. 
The sum $\by_k = \sum_j w_j(\bz_k - \mathcal{P}_{Q_j}(\bz_k))$ requires $\mathcal{O}(m)$ flops. The matrix-vector product $\bA^t\by_k$ requires $\mathcal{O}(mn)$ flops. Adding the two sums requires $\mathcal{O}(n)$ flops. In total, the gradient requires $\mathcal{O}(mn)$ flops.

We next tally the flops required to compute the MM update. Forming the matrix $v\bI + w\bA^t\bA$ requires $\mathcal{O}(mn^2)$ flops. Computing its Cholesky factorization requires $\mathcal{O}(n^3)$ flops. We only need to compute the factorization once and can cache it. Subsequent iterations will require $\mathcal{O}(n^2)$ flops.
Thus, the amount of work required to compute a single MM update for the linear formulation is $\mathcal{O}(mn)$.

Note that the exact update in the linear case given by \[ \bx_{k+1}  =  (v\bI + w\bA^t\bA)^{-1} \left[ \sum_i \bv_i \mathcal{P}_{C_i}(\bx_{k}) + \bA^t \sum_j w_j \bA \mathcal{P}_{Q_j}(\bA\bx_{k}) \right] \]
has the same complexity as the update considered above.

\subsubsection{Non-linear}
We next consider the number of flops required for an MM update in the non-linear case:
\begin{eqnarray*}
\nabla f(\bx_{k}) & = & \sum_i v_i(\bx_{k} - \mathcal{P}_{C_i}(\bx_{k}))
+ \nabla h(\bx_{k})\sum_j w_j (h(\bx_{k}) - \mathcal{P}_{Q_j}(h(\bx_{k}))) \\
\bx_{k+1} & = & \bx_{k} - \eta_{\bx_{k}}[v\bI + w\nabla h(\bx_{k})dh(\bx_{k})]^{-1}\nabla f(\bx_{k}) \\
& = & \bx_{k} - \eta_{\bx_{k}}\left[\frac{1}{v}\bI - \frac{w}{v^2}\nabla h(\bx_{k}) \left (\bI + \frac{w}{v}dh(\bx_{k})\nabla h(\bx_{k}) \right)^{-1} dh(\bx_{k}) \right ]\nabla f(\bx_{k}) \\
& = & \bx_{k} - \eta_{\bx_{k}}\left[\frac{1}{v}\nabla f(\bx_{k}) - \frac{w}{v}\nabla h(\bx_{k}) \left (v\bI + wdh(\bx_{k})\nabla h(\bx_{k}) \right)^{-1} dh(\bx_{k})\nabla f(\bx_{k}) \right ].
\end{eqnarray*}

We first tally the flops required to compute the gradient $\nabla f(\bx_{k})$.
The first sum $\sum_i v_i(\bx_{k} - \mathcal{P}_{C_i}(\bx_{k}))$ requires $\mathcal{O}(n)$ flops. 
Computing $\bA\bx_{k}$ requires $\mathcal{O}(mn)$ flops. Computing the sum $\sum_j w_j (h(\bx_{k}) - \mathcal{P}_{Q_j}(h(\bx_{k})))$ requires $\mathcal{O}(m)$ flops. Computing $\nabla h(\bx_{k})$ requires $\mathcal{O}(mn)$ flops, and computing its product with the sum term $\nabla h(\bx_{k})\sum_j w_j (h(\bx_{k}) - \mathcal{P}_{Q_j}(h(\bx_{k})))$ requires $\mathcal{O}(mp)$ flops. In total, the gradient requires $\mathcal{O}(mn)$ flops. As in the linear case, the dominant cost are matrix-vector products involving the matrix $\bA$.

We next tally the flops required to compute the MM update. Forming the matrix $v\bI + w\nabla h(\bx_{k})dh(\bx_{k})$ requires $\mathcal{O}(np^2)$ flops. Computing its Cholesky factorization requires $\mathcal{O}(p^3)$ flops. Computing $dh(\bx_{k})\nabla f(\bx_{k})$ requires $\mathcal{O}(np)$ flops. Computing 
$\left (v\bI + wdh(\bx_{k})\nabla h(\bx_{k}) \right)^{-1} dh(\bx_{k})\nabla f(\bx_{k})$ requires $\mathcal{O}(p^2)$ flops. The product 
$\nabla h(\bx_{k})\left (v\bI + wdh(\bx_{k})\nabla h(\bx_{k}) \right)^{-1} dh(\bx_{k})\nabla f(\bx_{k})$ requires $\mathcal{O}(np)$ flops. Computing a candidate update requires $\mathcal{O}(n)$ flops.
%\begin{eqnarray*}
%\frac{1}{v}\nabla f(\bx_{k}) - \frac{w}{v}\nabla h(\bx_{k}) \left (v\bI + wdh(\bx_{k})\nabla h(\bx_{k}) \right)^{-1} dh(\bx_{k})\nabla f(\bx_{k})
%\end{eqnarray*}
%requires $\mathcal{O}(n)$ flops.
An objective function evaluation requires $\mathcal{O}(mn)$ flops. Thus, including the line search, the amount of work required to compute a single MM update for the linear formulation is $\mathcal{O}(\max\{mn,np^2\})$. When $p^2 < m$, we conclude that the computational work for a single MM update for the non-linear case is comparable to that of the linear case. In practice, reducing the problem size via a non-linear formulation may additionally reduce the number of MM 
updates. %Further, that the number of backtracking steps is small.

\bibliographystyle{myplain}
\bibliography{SplitFeasibilityMM}

\end{document}